\documentclass[10pt,reqno]{article}

\usepackage{graphicx,amsmath,amssymb,setspace}
\usepackage{srcltx}
\usepackage{xcolor}
\usepackage{amsthm}

\newtheorem{thm}{Theorem}[section]
\newtheorem{lemma}{Lemma}[section]
\newtheorem{remark}{Remark}[section]
\newtheorem{cor}{Corollary}[section]
\newtheorem{defin}{Definition}[section]
\newtheorem{prop}{Proposition}[section]

\numberwithin{equation}{section}

\newcommand{\mC}{\mathcal{C}}

\newcommand{\E}{\mathbb{E}}
\newcommand{\R}{\mathbb{R}}
\renewcommand{\P}{\mathbb{P}}
\newcommand{\mF}{\mathcal{F}}

\newcommand \mR {\mathcal{R}}
\newcommand \fR {\mathfrak{R}}
\newcommand  \N {\mathbb{N}}

\newcommand \tpsi {\tilde \psi}

\newcommand \tR {\tilde R}

\newcommand \bo {b_1}
\newcommand \bt {b_2}
\newcommand \tbo {\tilde b_1}
\newcommand \tbt {\tilde b_2}

\newcommand \opsi {\overline{\psi}}
\newcommand \upsi {\underline{\psi}}

\newcommand \tX {\tilde X}

\newcommand{\la} {\lambda}

\newcommand \del {\delta}
\newcommand \tet {\theta}

\newcommand \kap {\kappa}
\newcommand \bet {\beta}

\newcommand \gam {\gamma}
\newcommand \gamo {\gam_1}
\newcommand \gamt {\gam_2}
\newcommand \var {\varphi}
\newcommand \veps {\varepsilon}

\newcommand \tgamo {\tilde \gam_1}
\newcommand \tgamt {\tilde \gam_2}

\newcommand \Phin {\Phi_n}
\newcommand \Sm {S_m}
\newcommand \Smn {S_{m, n}}
\newcommand \xmn {x_{m, n}}
\newcommand \ymn {y_{m, n}}
\newcommand \qm {q_{m}}

\newcommand \hR {\widehat R}
\newcommand \cR {\check R}

\allowdisplaybreaks[4]

\title{Minimizing the Discounted Probability of \break Exponential Parisian Ruin via Reinsurance}

\author{Xiaoqing Liang%
\thanks{Corresponding author. Department of Statistics, School of Sciences, Hebei University of Technology, Tianjin 300401, P.\ R.\ China, liangxiaoqing115@hotmail.com. X.\ Liang thanks the National Natural Science Foundation of China (11701139, 11571189) and the Natural Science Foundation of Hebei Province (A2018202057) for financial support.}
\and Virginia R. Young%
\thanks{Department of Mathematics, University of Michigan, Ann Arbor, Michigan, 48109, vryoung@umich.edu. V.\ R.\ Young thanks the Cecil J. and Ethel M. Nesbitt Professorship of Actuarial Mathematics for financial support.}
}

\date{\today}

\begin{document}
\maketitle

\begin{abstract}
We study the problem of minimizing the discounted probability of {\it exponential Parisian ruin}, that is, the discounted probability that an insurer's surplus exhibits an excursion below zero in excess of an exponentially distributed clock.  The insurer controls its surplus via reinsurance priced according to the mean-variance premium principle, as in Liang, Liang, and Young \cite{LLY2019}.  We, first, find the optimal reinsurance strategy for a diffusion approximation of the classical Cram\'er-Lundberg risk model.  Then, we consider the classical risk model itself and apply stochastic Perron's method, as introduced by Bayraktar and S\^irbu \cite{BS2012, BS2013, BS2014}, to show that the minimum discounted probability of exponential Parisian ruin is the unique viscosity solution of its Hamilton-Jacobi-Bellman equation with boundary conditions at $\pm \infty$.

\medskip

{\bf Keywords.}  Exponential Parisian ruin; Cram\'er-Lundberg risk model; diffusion approximation; optimal reinsurance; mean-variance premium principle.

\medskip

{\bf AMS 2010 Subject Classification.}    93E20, 91B30, 47G20, 45J05.

\medskip

{\bf JEL Codes.}  C61, D81, G22.

\end{abstract}

\section{Introduction}

In classical risk theory, the probability of ruin has been a fundamental measure of the insurer's overall risk.  However, as discussed in dos Reis \cite{dR1993}, sometimes the probability of ruin is extremely small, or the portfolio might be one out of many existing businesses in the company.  There is a possibility that the insurer has extra funds (or can borrow money) to maintain the negative surplus until the portfolio recovers and before regulators detect the insurer's bankruptcy.   With this in mind, researchers introduced the concept of Parisian ruin, which arises from Parisian options; see Chesney, Jeanblanc-Picqu\'e, and Yor \cite{CJY1997}. (Ordinary) Parisian ruin occurs if an excursion below zero is longer than a deterministic time. Dassios and Wu \cite{DW2008} first extended the concept of ruin to this Parisian type of ruin and computed it for a classical risk model with exponential claims and for a diffusion approximation of the classical risk model.  Recently, Parisian ruin has been actively studied by, for example, Czarna and Parmowski \cite{CP2011}, Loeffen, Czarna, and Palmowski  \cite{LCP2013}, Guerin and Renaud \cite{GR2017}, Czarna and Palmowski \cite{CP2014}, Baurdoux, Pardo, P{\'e}rez, and Renaud \cite{BPPR2016}, and Renaud \cite{R2019}.  Most of this research focuses on calculating the probability of Parisian ruin via excursion theory of spectrally negative L{\'e}vy processes and their associated scale functions; none of the above-mentioned research controls the probability of Parisian ruin.  By contrast, Liang and Young \cite{LY2019} studied the stochastic control problem of optimally investing to minimize the probability of lifetime exponential Parisian ruin by solving a boundary-value problem formed by the problem's associated Hamilton-Jacobi-Bellman (HJB) equation and boundary conditions.

Stochastic Perron's method was introduced in Bayraktar and S{\^i}rbu \cite{BS2012} for linear problems, Bayraktar and S{\^i}rbu \cite{BS2013} for nonlinear problems of HJB equations in stochastic control, and Bayraktar and S{\^i}rbu \cite{BS2014} for problems related to Dynkin games. Stochastic Perron's method provides a way to show that the value function of the stochastic control problem is the unique viscosity solution of the associated HJB equation. Unlike the classical verification approach, it requires neither the dynamic programming principle nor the regularity of the value function.  Roughly speaking, stochastic Perron's method consists of the following steps: (1) estimate the value function from below and above by stochastic sub- and supersolutions, (2) prove that  the supremum and the infimum of the respective families are viscosity super- and subsolutions, respectively, and (3) prove a comparison principle for viscosity sub- and supersolutions, which has an immediate corollary that the value function is the unique (continuous) viscosity solution of its HJB equation.  As opposed to the classical verification method, the comparison principle here plays a role for both verification and uniqueness.  More recently, stochastic Perron's method was applied to solve an exit-time problem in Rokhlin \cite{R2014}, a transaction-cost problem in Bayraktar and Zhang \cite{BZ2015b}, stochastic target problems in Bayraktar and Li \cite{BL2016a, BL2016b, BL2017}, and a problem of optimal consumption and information in Yang and Yu \cite{YY2019}.

In this paper, we consider the problem of minimizing the discounted probability of exponential Parisian ruin for an insurance company who can purchase per-loss reinsurance.  We model the time that a regulator inspects the insurance company's accounting records (after bankruptcy occurs) via an exponentially distributed random variable with mean $1/\rho$.  {\it Exponential Parisian ruin} occurs if bankruptcy occurs, and if the company is still in bankruptcy when the regulator inspects the records.  In other words, as soon as the surplus process becomes negative, an independent exponential alarm clock begins to run; if the alarm rings before the surplus becomes positive again, then exponential Parisian ruin occurs.  Mathematically, our problem generalizes the ordinary-ruin problem and retains its one-dimensional property; specifically, because an exponential random variable is memoryless (that is, has a constant hazard rate), we do not need to introduce a time variable.

We assume that the insurer is allowed to purchase per-loss reinsurance, which is priced according to the mean-variance premium principle, as in Han, Liang, and Young \cite{HLY2019} and Liang, Liang, and Young \cite{LLY2019}.  We, first, find an explicit expression for the optimal reinsurance strategy for the diffusion approximation of the classical Cram\'er-Lundberg risk model.  Then, we consider the classical model itself, for which we cannot find an explicit expression of the minimum discounted probability of exponential Parisian ruin.  Instead, we use stochastic Perron's method to prove that the value function is the unique (continuous) viscosity solution of its associated discontinuous HJB equation with boundary conditions (Theorem \ref{thm:value}).

We extend the results of Bayraktar and S{\^i}rbu \cite{BS2013} to a problem with controlled Poisson jumps.  Most of the existing literature studies stochastic Perron's method under a diffusion model.  The major difficulty of the jump process comes from the fact that the surplus process might jump outside a small neighborhood, and the local results obtained from the viscosity sub- and supersolution property cannot be used.  Bayraktar and Zhang \cite{BZ2015b} also faced a similar problem, but the jump of their model occurs from singular control; to show the viscosity supersolution property of one of their bounds, they overcame the difficulty by splitting the jump into two parts--first to a point on the boundary of the neighborhood and, then, to its finally destination.  For the viscosity subsolution property of the other bound, they chose the control to be identically zero.  Another related work is Bayraktar and Li \cite{BL2016b}, who considered stochastic target problems with jump diffusions.  By applying a result in Bouchard and Dang \cite{BD2012}, namely, that a stochastic control problem can be converted to a stochastic target problem with unbounded controls, one could apply the results of Bayraktar and Li \cite{BL2016b} to our problem.  However, we provide direct proofs of the necessary viscosity properties (Theorems \ref{thm:v_plus} and \ref{thm:u_minus}) instead of relying on their indirect result.  More importantly, Bayraktar and Li \cite{BL2016b} took the comparison principle as an assumption and did not prove it; by contrast, we prove the comparison principle for our problem.

Another contribution of our paper is that we give a complete proof of the comparison principle for our specific control problem (Theorem \ref{thm:comp}). The challenge in the proof comes from the fact that the Hamiltonian in our model is discontinuous at zero, which means that we cannot apply the standard proof of doubling the variables. To overcome this difficulty, we borrow an idea from Giga, G\'orka, and Rybka \cite{GGR2011}. First, we define a continuous Hamiltonian, which approximates the discontinuous one.  Then, we verify that the viscosity supersolution of the discontinuous Hamiltonian is also a viscosity supersolution of the continuous, approximate Hamiltonian.  Next, we construct a new viscosity subsolution, which approximates the original one, for the continuous Hamiltonian.  Finally, we prove a comparison principle for the continuous, approximate Hamiltonian with unbounded domain and non-local jump term.  By taking the limit of resulting inequality between the new viscosity subsolution and the original supersolution, we obtain a comparison principle for our original sub- and supersolutions.

Two related papers also rely on viscosity solutions.  Azcue and Muler \cite{AM2005} explored an optimal dividend and reinsurance problem for the classical risk model, and they control the jump process via reinsurance, as we do in this paper.  Barles and Imbert \cite{BI2008} considered a second-order elliptic integro-differential equation.  However, the Hamiltonians in their models are continuous, so we cannot rely on their work to prove the comparison principle for our problem.

The remainder of this paper is organized as follows. In Section 2, we describe the reinsurance market and set up the problem of minimizing the discounted probability of exponential Parisian ruin.  Section 3 provides the explicit solution for the minimum discounted probability of exponential Parisian ruin for the diffusion approximation of the classical risk model. Section 4 is devoted to studying the minimum discounted probability of exponential Parisian ruin for the classical risk model.  We use stochastic Perron's method to prove that the value function is the unique (continuous) viscosity solution of its associated discontinuous HJB equation.  In Section 4.1, we consider the analog of the adjustment coefficient and construct a function that is a stochastic supersolution.  To apply stochastic Perron's method, in Section 4.2, we define stochastic supersolutions and prove that the infimum of all stochastic supersolutions is a viscosity subsolution. In Section 4.3, we define stochastic subsolutions and prove that the supremum of all stochastic subsolutions is a viscosity supersolution.  In Section 4.4, we prove the comparison principle for our problem, which ensures that a viscosity subsolution is smaller than a viscosity supersolution.  Hence, we obtain our conclusion.

\section{Reinsurance framework}\label{sec:fin-model}

In this section, we describe the reinsurance market available to the insurance company, and we formulate the problem of minimizing the discounted probability of exponential Parisian ruin.  Assume that all random processes exist on the filtered probability space $\big( \Omega, \mathcal{F}, \mathbb{F} = \{ F_t \}_{t \ge 0}, \mathbb{P} \big)$.

We model the insurer's claim process $C = \{C_t\}_{t \ge 0}$ according to a compound Poisson process, namely,
\begin{equation}
\label{eq:CPP}
C_t = \sum_{i = 1}^{N_t} Y_i,
\end{equation}
in which the claim severities $Y_1, Y_2, \dots$ are independent and identically distributed according to a common cumulative distribution function $F_Y$, with $Y > 0$ a.s., and in which the claim frequency $N = \{ N_t \}_{t \ge 0}$ follows a Poisson process with parameter $\la > 0$.  Let $S_Y = 1 - F_Y$ denote the survival function of $Y$, and assume that $\E Y < \infty$ and $\E \big( Y^2 \big) < \infty$.  Also, assume that the insurer receives premium payable continuously at a rate $c > \la \E Y$, and assume that the Poisson process $N$ is independent of the claim severity process $Y$.

The insurer can buy per-loss reinsurance; let $R_t(y)$ denote the retained claim at time $t \ge 0$, as a function of the (possible) claim $Y = y$ at that time.  Thus, reinsurance indemnifies the insurer with the amount $y - R_t(y)$ if there is a claim $y$ at time $t \ge 0$.   The reinsurance premium is continuously payable computed according to the mean-variance principle with non-negative risk loadings $\tet$ and $\eta$, that is, the reinsurance premium rate at time $t$ equals
\begin{equation}\label{eq:reins_prem}
(1 + \tet) \la \E \big( Y - R_t(Y) \big) + \dfrac{\eta}{2} \, \la \E \big( (Y - R_t(Y))^2 \big).
\end{equation}
Assume that
\begin{equation}\label{eq:c_not_huge}
c < (1 + \tet) \la \E Y + \dfrac{\eta}{2} \, \la \E \big(Y^2 \big);
\end{equation}
in words, the insurer's premium income is not sufficient to buy full reinsurance.  A retention strategy $\mR = \{ R_t \}_{t \ge 0}$ is {\it admissible} if for fixed $y$, the mapping $(t,w) \mapsto R_t(\omega, y)$ is $\mathbb{F}$-predictable, and for fixed $(t,w)$, $R_t(\omega, y)$ is
 $\mathcal{B}(\R^+)$-measurable, in which $\mathcal{B}(\R^+)$ denotes the Borel $\sigma$-algebra on $\R^+$,\footnote{We will generally drop the dependence on $\omega$ when writing $R_t = R_t(y)$.} and satisfies $0 \le R_t(y) \le y$, for all $t \ge 0$.  Denote the set of admissible strategies by $\fR$.

\begin{remark}\label{rem:moralhazard}
Our assumptions concerning the retention strategy $\mR$ are designed to avoid moral hazard. Indeed, if $R_t(y) < 0$ for some values of $t$ and $y$, then the insurer would have an incentive to create a loss of $y$ at time $t$ to obtain a payment $-R_t(y) > 0$ from the reinsurance company. If $R_t(y) > y$, then the insurer would be acting like a reinsurer, and we wish to keep the roles of insurer and reinsurer separate.

Also, the reinsurance to be paid at time $t$ is chosen before the possible claim at time $t$ occurs.  Otherwise, the insurer would change its retention to $0$ immediately when a claim occurs.  \qed

\end{remark}

Given a retention strategy $\mR \in \fR$, the insurer's surplus follows the dynamics
\begin{align}
\label{eq:X}
dX^\mR_t &= \left( c - (1 + \tet) \la \E \big( Y - R_t(Y) \big) - \dfrac{\eta}{2} \, \la \E \big( (Y - R_t(Y))^2 \big) \right)dt - R_t(Y) dN_t \notag \\
&= \left( - \kap + \la \left( (1 + \tet) \E R_t(Y) + \eta \E \big( Y R_t(Y) \big) - \dfrac{\eta}{2} \, \E \big( R^2_t(Y) \big) \right) \right)dt - R_t(Y) dN_t,
\end{align}
in which $\kap$ is the positive constant
\begin{equation}\label{eq:kap}
\kap = (1 + \tet) \la \E Y + \dfrac{\eta}{2} \, \la \E \big(Y^2 \big) - c.
\end{equation}

By {\it exponential Parisian ruin}, we mean an excursion of surplus below zero in excess of a random length of time $\tau_\rho$, in which $\tau_\rho$ is exponentially distributed with hazard rate $\rho$, that is, with $\E(\tau_\rho) = 1/\rho$.  We assume that $\tau_\rho$ is independent of the claim process $C$.  One could also consider an excursion of surplus below some arbitrary level, not necessarily $0$, but for ease of presentation, we choose the level to be $0$.  Following Gu\'erin and Renaud \cite{GR2017}, we define exponential Parisian ruin by
\begin{equation}\label{eq:expPar_ruin}
K = \inf \{t > 0: t - g_t > \tau_\rho \},
\end{equation}
in which $g_t = \sup \{ s \in [0, t]: X_s \ge 0 \}$ with $\sup \emptyset = 0$.\footnote{More technically accurate, in place of $\tau_\rho$, we have a sequence $\tau_\rho^1, \tau_\rho^2, \dots$, of i.i.d.\ random variables, each of which begins when the surplus is newly negative.}  Note that we reset the ``excursion clock'' to $0$ whenever surplus reaches $0$ from below; indeed, if $X_t \ge 0$, then $t - g_t = 0$, and $K = \inf \emptyset = \infty$.  We wish to minimize the discounted probability of exponential Parisian ruin, with value function
\begin{equation}\label{eq:psi}
\psi(x) = \inf_{\mR \in \fR} \E^x \Big( e^{-\bet K} {\bf 1}_{\{K < \infty\}} \Big),
\end{equation}
in which $\E^x$ denotes expectation conditional on $X_0 = x$, and $\bet > 0$ measures the insurer's time value of exponential Parisian ruin.  If exponential Parisian ruin were to occur at a distant time in the future, then the insurer would be less unhappy than if ruin were to occur today.

Clearly, the value function $\psi$ is non-increasing on $\R$, with $\lim \limits_{x \to -\infty} \psi(x) = \rho/(\rho + \bet)$ and $\lim \limits_{x \to \infty} \psi(x) = 0$.  Also, $\psi$ is non-decreasing with respect to $\rho$, the hazard rate for the exponential Parisian clock, and as $\rho \to \infty$, we expect the value function to approach the discounted probability of ordinary ruin.  At the other extrem, as $\rho \to 0$, exponential ruin cannot occur, and $\psi$ converges to $0$ on $\R$.  Moreover, we informally prove the following lemma, which will be useful in Section \ref{sec:CPP}.

\begin{lemma}\label{lem:psi_pos}
For $x \in \R$, we have the strict inequality
\begin{equation}\label{eq:psi_pos}
\psi(x) > 0.
\end{equation}
\end{lemma}

\begin{proof}
We prove this lemma for $\bet = 0$; $\bet > 0$ will not change the positivity of $\psi$.  In Section 2.3.1 of \cite{S2008}, Schmidli proves that the minimum probability of ordinary ruin, controlled by reinsurance subject to a mathematically suitable premium rule (for example, the mean-variance premium principle), is strictly decreasing on $\R^+$; thus, it is strictly positive on $\R^+$.  For our problem, exponential Parisian ruin requires, first, that the surplus process become negative (that is, ordinary ruin occurs); then, while the surplus is negative, the alarm in the exponential Parisian clock must ring, intuitively speaking.

Fix a retention strategy $\mR$, and let $\psi^\mR_0$ and $\psi^\mR$ denote the probability of ordinary ruin and exponential Parisian ruin, respectively, under the retention strategy $\mR$.  Also, let $Z$ denote the random variable of the deficit at ruin with cumulative distribution function $F_Z$, which might depend on $X^\mR_0 = x$.  Then, if we consider the first excursion of surplus into the negative reals, we obtain, for $x \ge 0$,
\begin{equation}\label{ineq:psi_psi0}
\psi^\mR(x) \ge \psi^\mR_0(x) \cdot \P \left[ \int_0^{\infty} \left(1 - e^{- \rho z/c} \right) dF_Z(z) \bigg| \tau_0 < \infty \right] > 0,
\end{equation}
in which $\tau_0 = \inf \{t \ge 0: X^\mR_t < 0\}$ is the time of ordinary ruin, and in which the second inequality follows from the positivity of $\psi^\mR_0$.  The integrand $1 - e^{- \rho z/c}$ equals the probability that the exponential Parisian alarm clock rings during the shortest possible time the insurer spends in this excursion, namely, $z/c$.  First, minimize $\psi^\mR_0$ over admissible retention strategies $\mR$, then minimize $\psi^\mR$ over admissible retention strategies to obtain
\[
\psi(x) \ge \psi_0(x) \cdot \P \left[ \int_0^{\infty} \left(1 - e^{- \rho z/c} \right) dF_Z(z) \bigg| \tau_0 < \infty \right] > 0,
\]
in which the second inequality follows from Section 2.3.1 of Schmidli \cite{S2008}.  Thus, $\psi$ is strictly positive on $\R^+$ (discounting for the time of ruin does not change this conclusion); moreover, because $\psi$ is non-increasing, it is strictly positive on all of $\R$.
\end{proof}

In Section \ref{sec:CPP}, we work with the classical risk model in \eqref{eq:X}; but, first, in Section \ref{sec:diff}, we obtain explicit results by considering the diffusion approximation of $X$. To obtain the diffusion approximation, we match the first two moments at all times $t \ge 0$, as in Grandell \cite{G1991}.  Specifically,
\[
R_t(Y) dN_t \approx \la \E R_t(Y) dt - \sqrt{\la \E \big(R^2_t(Y) \big)} \, dB_t,
\]
in which $B = \{ B_t \}_{t \ge 0}$ is a standard Brownian motion on $\big( \Omega, \mathcal{F}, \mathbb{F}, \mathbb{P} \big)$.
Thus, in Section \ref{sec:diff}, we assume that surplus approximately follows the process $\tX = \{ \tX_t \}_{t \ge 0}$ with dynamics
\begin{align}
\label{eq:Xtilde}
d \tX_t &= \left( - \kap + \la \left( \tet \E R_t(Y) + \eta \E \big( Y R_t(Y) \big) - \dfrac{\eta}{2} \, \E \big( R^2_t(Y) \big) \right) \right)dt + \sqrt{\la \E \big(R^2_t(Y) \big)} \, dB_t .
\end{align}
Let $\tilde K$ and $\tpsi$ denote the corresponding exponential Parisian ruin and minimum discounted probability thereof, respectively.

\section{Discounted exponential Parisian ruin: diffusion approximation}\label{sec:diff}

We begin by stating a relevant verification theorem without proof because its proof is standard in the actuarial and financial mathematics literature; see, for example, the proof of Theorem 3.1 in Han, Liang, and Young \cite{HLY2019}.

\begin{thm}\label{thm:verif_diff}
Suppose $v$$: \R \to \left[0, \, \frac{\rho}{\rho + \bet} \right]$ is a function that satisfies the following conditions.
\begin{enumerate}
\item{} $v$ is decreasing and lies in $\mC^2(\R),$ except at $0,$ where it is $\mC^1$ and has left- and right-second derivatives.
\item{} $\lim \limits_{x \to -\infty} v(x) = \dfrac{\rho}{\rho + \bet} \,$.
\item{} $\lim \limits_{x \to \infty} v(x) = 0$.
\item{} $v$ solves the following Hamilton-Jacobi-Bellman $($HJB$)$ equation on $\R$$:$
\begin{equation}\label{eq:HJB_diff}
\bet v + \rho( v - 1) {\bf 1}_{\{x < 0\}} = - \kap v_x + \la \inf \limits_R \bigg[ \left( \tet \E R  + \eta \E \big( Y R \big) - \dfrac{\eta}{2} \, \E \big( R^2 \big) \right) v_x + \frac{1}{2} \, \E \big( R^2 \big) v_{xx} \bigg],
\end{equation}
\end{enumerate}
in which we take the infimum over retention functions $R$ such that $0 \le R(y) \le y$ for all $y \in \R^+$.  Then, the minimum discounted probability of exponential Parisian ruin $\tpsi$ equals $v$, and the optimal retention strategy $\tilde \mR$ is given in feedback form by the {\rm arg min} of \eqref{eq:HJB_diff}. \qed
\end{thm}

Note that the only dependence of the HJB equation in \eqref{eq:HJB_diff} upon the state variable $x$ is whether $x < 0$ or $x \ge 0$.  Based on previous work in goal-seeking problems with diffusion dynamics, we hypothesize that the value function is of the form
\begin{equation}\label{eq:hypoth}
v(x) =
\begin{cases}
\dfrac{\rho}{\rho + \bet} - \tbt e^{\tgamt x}, &\quad x < 0, \vspace{1.5ex} \\
\tbo e^{-\tgamo x}, &\quad x \ge 0,
\end{cases}
\end{equation}
for some positive constants $\tbo$, $\tbt$, $\tgamo$, and $\tgamt$.  If, for $x < 0$, $\tpsi$ is of the form given in \eqref{eq:hypoth}, then it is straightforward to show that the optimal retention strategy is full retention, that is, $\tR_t(y) = y$ for all $t \ge 0$, and $\tgamt$ equals the unique positive root of
\begin{equation}\label{eq:tgamt}
\dfrac{\la}{2} \, \E \big(Y^2 \big) \gam^2 + \big( c - \la \E Y \big) \gam - (\rho + \bet) = 0.
\end{equation}

If, for $x \ge 0$, $\tpsi$ is of the form given in \eqref{eq:hypoth}, then by adapting the proof of Theorem 3.1 of Li and Young \cite{LiY2019} or the proof of Lemma 3.1 of Liang, Liang, and Young \cite{LLY2019}, one can show that the optimal retention function is given by
\begin{equation}\label{eq:tR1}
\tR(y) = \dfrac{\tet + \eta y}{\eta + \tgamo} \wedge y,
\end{equation}
in which $\tgamo > 0$ uniquely solves
\begin{equation}\label{eq:tgamo}
\tet \E \tR  + \eta \E \big(Y\tR \big) - \dfrac{\eta + \gam}{2} \, \E \big(\tR^2 \big) = \dfrac{\kap \gam - \bet}{\la \gam} \, ,
\end{equation}
with $\kap$ given in \eqref{eq:kap}, or  equivalently,
\begin{equation}\label{eq:tgamo_2}
(c - \la \E Y) + \dfrac{\bet}{\gam} = \la \gam \int_0^\infty \left( \dfrac{\tet + \eta y}{\eta + \gam} \wedge y \right) S_Y(y) dy .
\end{equation}
It remains to obtain $\tbo$ and $\tbt$ via smooth pasting at $x = 0$, and we give that solution in the following theorem.

\begin{thm}\label{thm:diff}
The minimum discounted probability of exponential Parisian ruin under the diffusion approximation equals
\begin{equation}\label{eq:tpsi}
\tpsi(x) =
\begin{cases}
\dfrac{\rho}{\rho + \bet} \left(1 - \dfrac{\tgamo}{\tgamo + \tgamt} \, e^{\tgamt x} \right), &\quad x < 0, \vspace{1.5ex} \\
\dfrac{\rho}{\rho + \bet} \, \dfrac{\tgamt}{\tgamo + \tgamt} \, e^{- \tgamo x}, &\quad x \ge 0,
\end{cases}
\end{equation}
in which $\tgamt > 0$ is given by
\begin{equation}\label{eq:gam2}
\tgamt = \dfrac{1}{\la \E \big(Y^2 \big)} \left[ \sqrt{\big( c - \la \E Y \big)^2 + 2 (\rho + \bet) \la \E \big(Y^2 \big)} - \big( c - \la \E Y \big) \right],
\end{equation}
and $\tgamo$ solves \eqref{eq:tgamo_2}.  The corresponding optimal retention strategy $\tilde \mR$ is given in feedback form by $\tR_t(y) = \tR(X_t, y)$ in which
\begin{equation}\label{eq:tR}
\tR(x, y) =
\begin{cases}
y, &\quad x < 0, \vspace{1.5ex} \\
\dfrac{\tet + \eta y}{\eta + \tgamo} \wedge y, &\quad x \ge 0.
\end{cases}
\end{equation}
\end{thm}

\begin{proof}
It is straightforward to show that $\tpsi$ in \eqref{eq:tpsi} satisfies the conditions of Theorem \ref{thm:verif_diff}; thus, it equals the minimum discounted probability of exponential Parisian ruin.  Moreover, $\tR$ in \eqref{eq:tR} is the minimizer of the HJB equation in \eqref{eq:HJB_diff}, which implies that it determines the optimal retention strategy.
\end{proof}

\begin{remark}
Theorem {\rm \ref{thm:diff}} shows that the insurer retains {\rm all} of its risk when the surplus is negative.  What drives this result is that the insurer is essentially maximizing its drift and volatility in order to make the surplus positive as soon as possible, in order to avoid exponential Parisian ruin.  \qed
\end{remark}

\begin{remark}
If we set $\eta = 0$, then the reinsurance premium principle reduces to the expected-value premium principle, and excess-of-loss reinsurance is optimal when $x \ge 0;$ see \eqref{eq:tR}. Similarly, if we set $\tet = 0$, then the premium principle reduces to the variance premium principle, and proportional reinsurance is optimal when $x \ge 0$.  \qed
 \end{remark}

\begin{remark}
As in Liang, Liang, and Young {\rm \cite{LLY2019}}, one can show that $\tgamo$ equals the analog of the maximum adjustment coefficient for the diffusion process when considering {\rm discounted} probability of ruin for $x \ge 0$.  When $x < 0$, $\tgamt$ equals the analog of the adjustment coefficient for $\tR(y) = y$, but $\tgamt$ is not the maximum because if $\tpsi(x)$ is of the form $1 - \tbt e^{-\tgamt x}$, then the optimal $\tR(y) = y$, and \eqref{eq:tgamt} then determines $\tgamt$.   \qed
\end{remark}

We end this section with a straightforward corollary of Theorem \ref{thm:diff}.

\begin{cor}\label{cor:diff_bet0}
If we set $\bet$ equal to $0$ in Theorem $\ref{thm:diff}$, then $\tpsi$ in \eqref{eq:tpsi} equals the minimum probability of exponential Parisian ruin under the diffusion approximation.  \qed
\end{cor}

\section{Discounted exponential Parisian ruin: classical risk model}\label{sec:CPP}

For the classical risk model, as opposed to the diffusion model, we cannot find an explicit expression for the minimum discounted probability of exponential Parisian ruin.  In this section, we apply stochastic Perron's method, created by Bayraktar and S\^irbu \cite{BS2013}, to  prove that the value function $\psi$ is the unique continuous viscosity solution of its HJB equation with appropriate boundary conditions.  Define the operator $F$ via its action on appropriately differentiable functions $u$, $v$, and $w$ as follows:
\begin{align}
\label{oper:F}
F \big(x, u(x), v_x(x), w(\cdot) \big)
& = \bet u(x) + \rho \big(u(x) - 1 \big) {\bf 1}_{\{x < 0\}} + \kap v_x(x) \notag \\
& \quad - \la \inf_{R} \left[ \left( (1 + \tet) \E R + \eta \E(YR) - \dfrac{\eta}{2} \, \E \big(R^2 \big) \right) v_x(x) + \textcolor[rgb]{0.00,0.00,0.00}{\E w(x - R) - w(x)} \right],
\end{align}
in which we take the infimum over retention functions $R$ such that $0 \le R(y) \le y$ for all $y \in \R^+$.  Then, the HJB equation for our problem is
\begin{align}
\label{HJB_F}
F\big(x, v(x), v_x(x), v(\cdot) \big) = 0,
\end{align}
with boundary conditions
\begin{align}
\label{boundary_condition}
\lim \limits_{x \to -\infty} v(x) = \dfrac{\rho}{\rho + \bet} \, , \qquad \qquad \lim \limits_{x \to \infty} v(x) = 0.
\end{align}
Note that $F$ is a discontinuous operator at $x = 0$.  

We, next, define viscosity sub- and supersolutions for our problem, taking into account $F$'s discontinuity at $x = 0$.

\begin{defin}\label{def:strict_glob}
We say an upper semi-continuous $($u.s.c.$)$ function $\underline{u}: \R \to \left[ 0, \, \frac{\rho}{\rho + \bet} \right]$ is a {\rm viscosity subsolution} of \eqref{HJB_F} and \eqref{boundary_condition} if
\[
\lim \limits_{x \to -\infty} \underline{u}(x) = \dfrac{\rho}{\rho + \bet} \, , \qquad \qquad \lim \limits_{x \to \infty} \underline{u}(x) = 0,
\]
and if, for any $x_0 \in \R$ and for any $\var \in \mC^1(\R)$ such that $\underline{u} - \var$ reaches a strict, global maximum of zero at $x_0$, we have
\begin{equation}\label{eq:vis_sub}
\begin{cases}
F\big(x_0, \var(x_0), \var_x(x_0), \var(\cdot)\big) \le 0, &\quad x_0 \ne 0, \\
\lim \limits_{x \to 0-} F\big(x, \var(x), \var_x(x), \var(\cdot) \big) \le 0, &\quad x_0 = 0.
\end{cases}
\end{equation}
Similarly, we say a lower semi-continuous $($l.s.c.$)$ function $\bar{u} : \R \to \left[ 0, \, \frac{\rho}{\rho + \bet} \right]$ is a {\rm viscosity supersolution} of \eqref{HJB_F} and \eqref{boundary_condition} if
\[
\lim \limits_{x \to -\infty} \bar{u}(x) = \dfrac{\rho}{\rho + \bet} \, , \qquad \qquad \lim \limits_{x \to \infty} \bar{u}(x) = 0,
\]
and if, for any $x_0 \in \R $ and for any $\phi \in \mC^1(\R)$ such that $\bar{u} - \phi$ reaches a strict, global minimum of zero at $x_0$, we have
\begin{equation}\label{eq:vis_super}
F\big(x_0, \phi(x_0), \phi_x(x_0), \phi(\cdot)\big) \ge 0.
\end{equation}
Finally, a function $u$ is called a $($continuous$)$ {\rm viscosity solution} of \eqref{HJB_F} and \eqref{boundary_condition} if it is both a viscosity subsolution and a viscosity supersolution of \eqref{HJB_F} and \eqref{boundary_condition}.  \qed
\end{defin}

\begin{remark}
Because $\var \in \mC^1(\R)$, the only discontinuity in $F$ in \eqref{eq:vis_sub} at $x = 0$ arises from the term $\rho(\var(x) - 1) {\bf 1}_{\{x < 0 \}}$.  Because we restrict viscosity sub- and supersolutions to take values in $\left[ 0, \, \frac{\rho}{\rho + \bet} \right]$, $F$ restricted to such functions is upper semi-continuous, and by taking the left limit in \eqref{eq:vis_sub} when $x_0 = 0$, we are essentially using the l.s.c.\ envelope of $F$.  See Ishii {\rm \cite{I1985}} for early work and Barles and Chasseigne {\rm \cite{BC2018}} for more recent work on viscosity solutions of discontinuous Hamiltonians.

We could allow non-strict equality in the boundary conditions, that is, $\lim \limits_{x \to -\infty} \underline{u}(x) \le \frac{\rho}{\rho + \bet}$, $\lim \limits_{x \to \infty} \underline{u}(x) \le 0$, $\lim \limits_{x \to -\infty} \bar{u}(x) \ge \frac{\rho}{\rho + \bet}$, and $\lim \limits_{x \to \infty} \bar{u}(x) \ge 0$, but it is useful in what follows to require strict equality.  \qed
\end{remark}

We use stochastic Perron's method, introduced by Bayraktar and S\^irbu \cite{BS2013}, to construct a solution of the HJB equation and, then, use a comparison lemma to verify that this solution equals the value function.  The main arguments are as follows:  First, we bound the value function from below and above by {\it stochastic} sub- and supersolutions (as defined in Sections \ref{sec:41} and \ref{sec:42}):
\begin{align}
\label{eq:u_vs_v}
u \le \psi \le v.
\end{align}
Let ${\Psi^{-}}$ and ${\Psi^{+}}$ denote the sets of stochastic sub- and supersolutions, respectively.  Define $u_{-}$ and $v_{+}$ on $\R$ by
\[
u_{-}(x) = \sup_{u \in \Psi^{-}} u(x), \qquad \qquad v_{+}(x) = \inf_{v \in \Psi^{+}} v(x).
\]
From \eqref{eq:u_vs_v}, we deduce
\begin{align*}
u_{-} \le \psi \le v_{+}.
\end{align*}
Second, we prove that $u_{-}$ is a viscosity supersolution and $v_{+}$ is a viscosity subsolution of \eqref{HJB_F} and \eqref{boundary_condition}.  Third, a comparison result for viscosity sub- and supersolutions implies the reverse inequality, namely,
\[
u_{-} \ge v_{+}.
\]
Thus, we conclude that $\psi (= u_{-} = v_{+})$ is the unique (continuous) viscosity solution of the HJB equation satisfying the boundary conditions in \eqref{boundary_condition}.  In the next subsection, we take a small diversion to consider the analog of the adjustment coefficient; then, in the subsequent three subsections, we work through the details of this outline.

\subsection{Analog of the adjustment coefficient}\label{sec:adj_coeff}

To find the analog of the adjustment coefficient, we split the problem according to whether $x \ge 0$ and $x < 0$. We begin with $x \ge 0$ because that is familiar to most readers.  A formal way of obtaining the adjustment coefficient is to substitute $\bo e^{-\gamo x}$ for $v$ in \eqref{HJB_F}, including when $x - R < 0$.  By doing so, we obtain the following equation for the (maximum) adjustment coefficient $\gamo > 0$:\footnote{It is not completely obvious that \eqref{eq:HJB_gamo} defines the {\it maximum} adjustment coefficient.  See Liang, Liang, and Young \cite{LLY2019} for the proof of this statement.}
\begin{equation}\label{eq:HJB_gamo}
0 = \bet - \kap \gamo + \la \sup \limits_{R} \left[ \left( (1 + \tet)\E R + \eta \E (YR) - \dfrac{\eta}{2} \, \E \big( R^2 \big) \right) \gamo - \big( M_R(\gamo) - 1 \big) \right],
\end{equation}
in which $M_R$ is the moment generating function of $R = R(Y)$.  Equation \eqref{eq:HJB_gamo} is similar to the equation for the maximum adjustment coefficient $\rho_J$ in Section 4.1 in Liang, Liang, and Young \cite{LLY2019}, and Theorem 4.1 in that paper gives us the optimal retention strategy corresponding to the maximum adjustment coefficient.  For completeness and future reference, we restate that proposition here, modified appropriately for our problem.

\begin{prop}\label{prop:gamo}
When $x \ge 0$, the maximum adjustment coefficient $\gamo > 0$ defined by \eqref{eq:HJB_gamo} uniquely solves
\begin{equation}\label{eq:gamo}
c + \dfrac{\bet}{\gam} = \la \int_0^\infty e^{\gam \hR(y; \gam)} S_Y(y) dy,
\end{equation}
and the corresponding optimal retention function $\hR$ is given by
\begin{equation}\label{eq:hR}
\hR(y; \gam) =
\begin{cases}
y, &\quad 0 \le y < \dfrac{1}{\gam} \ln(1 + \tet), \vspace{1ex} \\
R_c(y; \gam), &\quad y \ge \dfrac{1}{\gam} \ln(1 + \tet),
\end{cases}
\end{equation}
in which $R_c(y; \gam) \in [0, y]$ uniquely solves
\begin{equation}\label{eq:Rc}
(1 + \tet) + \eta y - \eta R - e^{\gam R} = 0,
\end{equation}
for $y \ge \frac{1}{\gam} \ln(1 + \tet)$ and for any $\gam > 0$.  \qed
\end{prop}

\medskip

Next, for $x < 0$, formally substitute $\frac{\rho}{\rho + \bet} - \bt e^{\gamt x}$ for $v$ in \eqref{HJB_F} to obtain the following equation for $\gamt > 0$:
\begin{equation}\label{eq:HJB_gamt}
\rho + \bet = - \kap \gamt + \la \sup \limits_{R} \left[ \left( (1 + \tet)\E R + \eta \E (YR) - \dfrac{\eta}{2} \, \E \big( R^2 \big) \right) \gamt - \big( 1 -  M_R(-\gamt) \big) \right].
\end{equation}
In the following proposition, we prove that the optimal retention function in \eqref{eq:HJB_gamt} equals full retention, which nicely matches the result we obtained for the diffusion approximation when $x < 0$.  Recall that admissible retention functions $R$ are restricted so that $0 \le R(y) \le y$.

\begin{prop}\label{prop:gamt}
When $x < 0$, the adjustment coefficient $\gamt > 0$ defined by \eqref{eq:HJB_gamt} uniquely solves
\begin{equation}\label{eq:gamt}
\rho + \bet = c \gamt - \la \big( 1 - M_Y( - \gamt) \big),
\end{equation}
and the corresponding optimal retention function is given by $\hR(y) = y$.
\end{prop}

\begin{proof}
Define the function $k$ by the expression in the square brackets of \eqref{eq:HJB_gamt}, ignoring the constant term $- 1$.  Specifically,
\begin{align}\label{eq:k}
k(R; \gam)&= \left((1 + \tet) \E R + \eta \E(Y R) - \frac{\eta}{2} \, \E\big(R^2 \big) \right) \gam  + M_R(-\gam) \notag  \\
&= \int_0^\infty \left[ \left((1 + \tet) R(y) + \eta y R(y) - \frac{\eta}{2} \, R^2(y) \right) \gam + e^{-\gam R(y)} \right] dF_Y(y).
\end{align}
For a given value of $\gam > 0$, we wish to find $\hR(y; \gam)$ that maximizes $k$.  Consider the integrand in the second line of the expression for $k$, namely,
\[
\ell(R) = \left((1 + \tet) R + \eta y R - \frac{\eta}{2} \, R^2 \right) \gam  + e^{-\gam R}.
\]
By differentiating with respect to $R$, we obtain
\[
\ell_R(R) = \big((1 + \tet) + \eta y - \eta R \big) \gam  - \gam e^{-\gam R},
\]
which is positive for all $R \in [0, y]$; thus, the optimal retention function $\hR$ is given by $\hR(y) = y$.

By substituting $\hR(y) = y$ for $R$ in \eqref{eq:HJB_gamt}, we obtain \eqref{eq:gamt}. It is straightforward to show that \eqref{eq:gamt} has a unique positive solution $\gamt$.
\end{proof}

\begin{remark}
Compare \eqref{eq:gamt} with equation $(4)$ in dos Reis {\rm \cite{dR1993}}. When $\bet = 0$, the two equations are identical with $\gamt = -f(-\rho)$, in which $f$ is given by the latter equation, just as $\gamt$ is given by \eqref{eq:gamt}.  Also, by performing integration by parts, we can rewrite \eqref{eq:gamt} to obtain an equation that is parallel to \eqref{eq:gamo}$:$
\[
\qquad \qquad \qquad \qquad \qquad \qquad \qquad c \, - \, \dfrac{\rho + \bet}{\gamt} = \la \int_0^\infty e^{-\gamt y} S_Y(y) dy. \qquad \qquad \qquad \qquad \qquad \qquad \qed
\]
\end{remark}

By analogy with the expression for $\tpsi$ in \eqref{eq:tpsi}, define $\opsi$ on $\R$ by
\begin{equation}\label{eq:opsi}
\opsi(x) =
\begin{cases}
\dfrac{\rho}{\rho + \bet} \left( 1 - \dfrac{\gamo}{\gamo + \gamt} \, e^{\gamt x} \right), &\quad x < 0, \vspace{1.5ex} \\
\dfrac{\rho}{\rho + \bet} \, \dfrac{\gamt}{\gamo + \gamt} \, e^{- \gamo x}, &\quad x \ge 0.
\end{cases}
\end{equation}
In the next section, we show that $\opsi$ is a stochastic supersolution of our problem.

\subsection{Stochastic supersolution}\label{sec:41}

To apply stochastic Perron's method, we first redefine the stochastic control problem using a stronger formulation.  To that end, let $0 \le \tau \le \omega \le K$ be stopping times.  Recall that $K$ is the time of exponential Parisian ruin, defined in \eqref{eq:expPar_ruin}.  Let $\fR_{\tau, \omega}$ denote the collection of predictable processes $\mR: (\tau, \omega] \to \R^+$, by which we mean that for fixed $y$, the mapping $(t, \varpi) \mapsto R_t(\varpi, y) \times {\bf 1}_{\{ \tau < t \le \omega \}}$ is predictable with respect to the filtration $\mathbb{F}$, and $0 \le R_t(y) \le y$ for all $t$ in the stochastic interval $(\tau, \omega]$.

\begin{defin}
A pair $(\tau, \zeta)$ is called a {\rm random initial condition} if $\tau$ is an $\mathbb{F}$-stopping time taking values in $[0, K]$ and $\zeta$ is an $\mF_{\tau}$-measurable random variable.  Then, for $\mR \in \fR_{\tau, K}$, the insurer's surplus process ${X}^{\tau, \zeta, \mR}$ is given by, for $t \in [\tau, K)$,
\begin{align}
\label{eq:X_rand}
X^{\tau, \zeta, \mR}_t
&= \zeta + \int^t_{\tau} \left( - \kap + \la \left( (1 + \tet) \E R_s + \eta \E \big( Y R_s \big) - \dfrac{\eta}{2} \, \E \big( R^2_s \big) \right) \right)d s - \int^t_{\tau} R_s dN_s.
\end{align}
\end{defin}

For convenience in what follows, we introduce a so-called {\it coffin state} $\Delta$, which represents the state when exponential Parisian ruin occurs.  We set $\Delta + x = \Delta$ for all $x \in \R$ and $X_t = \Delta$ for all $ t \in [K, \infty)$.  For any function $u$ defined on $\R$, we extend it to $\R \cup \{ \Delta \}$ by setting $u(\Delta) = 1$.  Next, we define a stochastic supersolution.

\begin{defin} \label{def:stoch_super}
A u.s.c.\ function $v: \R \to \left[0, \, \frac{\rho}{\rho + \bet} \right]$ is called a {\rm stochastic supersolution} if it satisfies the following properties:
\begin{itemize}
\item[$(1)$]
For any random initial condition $(\tau, \zeta)$, there exists a retention strategy $\mR \in \fR_{\tau, K}$ such that, for any
$\mathbb{F}$-stopping time $\omega \in [\tau, K]$,
\[
e^{-\bet \tau} v(\zeta) \ge \E \left[e^{-\bet \omega} v \big({X}^{\tau, \zeta, \mR}_\omega \big) \Big | \, \mF_{\tau}\right] ~~ a.s,
\]
in which $v$ is understood to be its extension to $\R \cup \{\Delta \}$.  We say that $\mR$ is {\rm associated with} $v$ for the initial condition $(\tau, \zeta)$.
\item[$(2)$] $\lim \limits _{x \to - \infty} v(x) = \dfrac{\rho}{\rho + \bet} \, $ and $\lim \limits _{x \to \infty} v(x) \ge 0$.
\end{itemize}
Let ${\Psi^+}$ denote the set of stochastic supersolutions.  \qed
\end{defin}

$\Psi^+$ is non-empty because $\frac{\rho}{\rho + \bet} \in \Psi^+$.  However, it is more useful to have a stochastic supersolution that satisfies the boundary conditions with equality; therefore, we present the following lemma.

\begin{lemma}\label{lem:opsi}
The function $\opsi$ defined in \eqref{eq:opsi} is a stochastic supersolution.
\end{lemma}

\begin{proof}
By construction, $\opsi$ is in $\mC^1(\R)$, and it decreases from $\frac{\rho}{\rho + \bet}$ to $0$ on $\R$; thus, $\opsi$ is u.s.c.\ and satisfies condition (2) in Definition \ref{def:stoch_super} with equality.  To show condition (1) of that definition, consider a random initial condition $(\tau, \zeta)$, and define the retention strategy $\check \mR = \{ \cR_t \}_{\tau \le t \le K}$ in feedback form by
\begin{equation*}
\cR_t(y) =
\begin{cases}
y, &\quad X_t < 0, \\
\hR(y), &\quad X_t \ge 0,
\end{cases}
\end{equation*}
in which $\hR$ is given in \eqref{eq:hR}.  For $x < 0$, the proof of Proposition \ref{prop:gamt} shows us that
\begin{equation}\label{eq:cR_xneg}
\bet \opsi(x) + \rho \big( \, \opsi(x) - 1 \big) + \left( \kap - \la \left( (1 + \tet) \E Y + \dfrac{\eta}{2} \, \E \big(Y^2 \big) \right) \right) \opsi_x(x) -  \la \big( \E \opsi(x - Y) - \opsi(x) \big) = 0.
\end{equation}
Furthermore, it is not difficult to show that, for $x \ge 0$,
\begin{equation}\label{eq:cR_xpos}
\bet \opsi(x) + \left( \kap - \la \left( (1 + \tet) \E \hR + \eta \E \big(Y \hR \big) - \dfrac{\eta}{2} \, \E \big( \hR^2 \big) \right) \right) \opsi_x(x) -  \la \left( \E \opsi \big(x - \hR \big) - \opsi(x) \right) \ge 0.
\end{equation}
Indeed, it follows from \eqref{eq:HJB_gamo} that inequality \eqref{eq:cR_xpos} is equivalent to
\begin{equation}\label{eq:ineq3}
\int_0^\infty \left[ \dfrac{\gamt}{\gamo + \gamt} \, e^{\gamo (\hR(y) - x)} + \dfrac{\gamo}{\gamo + \gamt} \, e^{- \gamt (\hR(y) - x)} - 1\right] {\bf 1}_{\{ \hR(y) > x \}} dF_Y(y) \ge 0.
\end{equation}
To show inequality \eqref{eq:ineq3}, it is enough to show
\[
\dfrac{b}{a + b} \, e^{az} + \dfrac{a}{a + b} \, e^{-bz} \ge 1,
\]
for all $z > 0$ and for $a$ and $b$ positive constants; proving this final inequality is a fun calculus exercise.

By applying a general version of It\^o's formula (see Protter \cite{P2005}) to $e^{- \bet \omega} \, \opsi \Big( X^{\tau, \zeta, \check \mR}_{\omega} \Big)$, we obtain
\begin{align}\label{eq:opsi_Ito}
&e^{- \bet \omega} \, \opsi \Big( X^{\tau,\zeta, \check \mR}_{\omega} \Big) \notag\\
&= e^{- \bet \tau} \, \opsi(\zeta) - \int_\tau^{\omega} e^{- \bet t} \bigg[ \bet \opsi \Big(X^{\tau, \zeta, \check \mR}_t \Big) + \rho \left( \opsi \Big(X^{\tau,\zeta, \check \mR}_t \Big) - 1 \right) {\bf 1}_{\{X^{\tau,\zeta, \check \mR}_t < 0 \}} \notag \\
&~~~\qquad \qquad \qquad \qquad \qquad + \left( \kap - \la \left( (1 + \tet) \E \cR_t + \eta \E(Y \cR_t) - \dfrac{\eta}{2} \, \E \big(\cR^2_t \big) \right) \right) \opsi_x \Big(X^{\tau, \zeta, \check \mR}_t \Big) \notag \\
&~~~\qquad \qquad \qquad \qquad \qquad  -  \la \Big( \E \opsi \Big(X^{\tau,\zeta, \check \mR}_t - \cR_t \Big) - \opsi \Big(X^{\tau,\zeta, \check \mR}_t \Big) \Big) \bigg] dt  \notag\\
&\quad + M_{\omega} + \int_\tau^{\omega} e^{- \bet t} \left(1- \opsi \Big(X^{\tau,\zeta, \check \mR}_t \Big) \right) {\bf 1}_{\{X^{\tau,\zeta, \check \mR}_t < 0 \}} d(N_t^{\rho} - \rho t),
\end{align}
in which $N^{\rho}$ is the jump process associated with exponential Parisian ruin, and
\begin{align*}
M_t = \sum \limits_{\substack{ X^{\tau,\zeta, \check \mR}_s \ne X^{\tau,\zeta, \check \mR}_{s-}\\ \tau \le s \le t }} e^{- \bet s} \left(\opsi \Big(X^{\tau,\zeta, \check \mR}_s \Big) - \opsi \Big(X^{\tau,\zeta, \check \mR}_{s-} \Big) \right) - \la \int_\tau^{t} e^{- \bet s} \left(\E \opsi \Big(X^{\tau, \zeta, \check \mR}_s - \cR_s \Big)
- \opsi \Big(X^{\tau,\zeta, \check \mR}_s \Big)\right) ds
\label{eq:M_t}
\end{align*}
is a martingale with zero $\mF_\tau$-expectation.  From \eqref{eq:cR_xneg} and \eqref{eq:cR_xpos}, we know that the integrand of the first integral is non-negative. Because $\opsi$ is bounded, the $\mF_\tau$-expectation of the second integral equals zero.  Thus, by taking the $\mF_\tau$-expectation of the expression in \eqref{eq:opsi_Ito}, we obtain
\[
e^{- \bet \tau} \, \opsi(\zeta) \ge \E \left[ e^{- \bet \omega} \, \opsi \Big({X}^{\tau,\zeta, \check \mR}_\omega \Big) \Big | \, \mF_{\tau}\right].
\]
Thus, $\opsi$ satisfies condition (1) in Definition \ref{def:stoch_super}, in which $\check \mR$ is associated with $\opsi$ for any initial condition $(\tau, \zeta)$.
\end{proof}

\begin{lemma}\label{lem:psi_v}
For any $v \in \Psi^+$, we have $\psi \le v$ on $\R$, that is, the minimum discounted probability of exponential Parisian ruin is a lower bound of any stochastic supersolution.
\end{lemma}

\noindent{\it Proof.}  First, note that $\psi \le v$ on the boundary of $\R$ by condition (2) in Definition \ref{def:stoch_super}.  Second, for $x \in \R$, let $(\tau, \zeta) = (0, x)$, and let $\mR$ be associated with $v$ for this initial condition.  By applying the supermartingale property (1) in Definition \ref{def:stoch_super} with $\omega = K$, the time of exponential Parisian ruin, and by recalling that $v(\Delta) = 1,$ we have
\begin{align*}
\qquad \qquad v(x) \ge \, \E \left[ e^{-\bet K} v \Big(X_{K}^{x, \mR} \Big) \right] \,
&= \E \left[ e^{-\bet K} v \Big(X_{K}^{x, \mR} \Big) {\bf 1}_{\{K < \infty\}} \right] + \E \left[ e^{-\bet K} v \Big(X_{K}^{x, \mR} \Big) {\bf 1}_{\{K = \infty\}} \right] \notag\\
& \ge \E \left[ e^{-\bet K} {\bf 1}_{\{K < \infty\}} \right] \ge \psi(x). \qquad\qquad \qquad \qquad \qquad \qquad \quad \qed
\end{align*}

\begin{lemma}\label{lem:sup-comb}
If $v_1$ and $v_2$ are two stochastic supersolutions, then $v = v_1 \wedge v_2$ is also a stochastic supersolution.
\end{lemma}

\noindent{\it Proof.}  We only need to verify item (1) of the definition of stochastic supersolution.  To that end, fix a random initial condition $(\tau, \zeta)$. Because $v_1$ and $v_2$ are two stochastic supersolutions, it follows that there are two controls $\mR_{i} \in
\fR_{\tau, K}$ for $i =1, 2$, such that
\begin{align*}
e^{- \bet \tau} v_i(\zeta) \ge \E\left[ e^{- \bet \omega} v_i \big({X}^{\tau,\zeta, \mR_{i}}_\omega \big) \Big | \, \mF_{\tau}\right]~~a.s.
\end{align*}
Define $A = \{v_1(\zeta) < v_2 (\zeta)\} \in \mF_{\tau}$, and define a retention strategy $\mR$ by
 \[
 \mR = \mR_1 {\bf 1}_A + \mR_2 {\bf 1}_{A^c}.
 \]
 Thus,
\begin{align*}
\qquad \qquad \E\left[ e^{- \bet \omega} v \Big({X}^{\tau,\zeta, \mR}_\omega \Big) \Big | \mF_{\tau}\right]
 & =  \E \left[ e^{- \bet \omega} v_1 \Big({X}^{\tau,\zeta, \mR_1}_\omega \Big)  {\bf 1}_A \Big | \mF_{\tau}\right]
 + \E \left[ e^{- \bet \omega} v_2 \Big({X}^{\tau,\zeta, \mR_2}_\omega \Big) {\bf 1}_{A^c} \Big | \mF_{\tau}\right] \\
 & \le  e^{- \bet \tau} v_1(\zeta) {\bf 1}_A + e^{- \bet \tau} v_2(\zeta) {\bf 1}_{A^c} = e^{- \bet \tau} v(\zeta) ~~a.s. \qquad \qquad \quad \;\; \qed
\end{align*}



\begin{thm}\label{thm:v_plus}
The upper stochastic envelope $v_{+}$, defined by
\begin{equation}\label{eq:v_plus}
v_{+}(x) = \inf \limits_{v \in {\Psi^+}} v(x),
\end{equation}
for $x \in \R$, is a viscosity subsolution of \eqref{HJB_F} and \eqref{boundary_condition}.
\end{thm}

\begin{proof}
$v_{+}$ takes values in $\left[ 0, \, \frac{\rho}{\rho + \bet} \right]$ and is u.s.c.\ because it is the pointwise infimum of a set of u.s.c.\ functions that take values in $\left[ 0, \, \frac{\rho}{\rho + \bet} \right]$.  $v_{+}$ satisfies the boundary conditions in \eqref{boundary_condition}, namely, $\lim \limits_{x \to - \infty} v_{+}(x) = \frac{\rho}{\rho + \bet}$ and $\lim \limits_{x \to \infty} v_{+}(x) = 0$.  The first limit follows from the definition of stochastic supersolution.  For the second limit, $\lim \limits_{x \to \infty} v_{+}(x) \ge 0$ from the definition of stochastic supersolution and from the definition of $v_{+}$ in \eqref{eq:v_plus}, $\lim \limits_{x \to \infty} \opsi(x) = 0$ from the definition of $\opsi$ in \eqref{eq:opsi}, and $v_{+} \le \opsi$ from Lemma \ref{lem:opsi}.

Next, we show the interior viscosity subsolution property.  Let  $x_0 \in \R$ and $\var \in \mC^1(\R)$ be such that $v_{+} - \var$ attains a strict, global maximum at $x_0$ with $v_{+}(x_0) = \var(x_0)$.  We need to show that
\begin{equation*}
\begin{cases}
F\big(x_0, \var(x_0), \var_x(x_0), \var(\cdot)\big) \le 0, &\quad x_0 \ne 0, \\
\lim \limits_{x \to 0-} F\big(x, \var(x), \var_x(x), \var(\cdot) \big) \le 0, &\quad x_0 = 0.
\end{cases}
\end{equation*}
Assume, on the contrary, that
\begin{equation*}
\begin{cases}
F\big(x_0, \var(x_0), \var_x(x_0), \var(\cdot)\big) > 0, &\quad x_0 \ne 0, \\
\lim \limits_{x \to 0-} F\big(x, \var(x), \var_x(x), \var(\cdot) \big) > 0, &\quad x_0 = 0.
\end{cases}
\end{equation*}
Then, by the continuity of $\var$ and $F$ (away from 0), by the strict maximization of $v_{+} - \var$ at $x = x_0$, and by $\var(x_0) = v_{+}(x_0) \ge \psi(x_0) > 0$,\footnote{Lemmas \ref{lem:psi_pos} and \ref{lem:psi_v} imply that $v_{+}(x_0) \ge \psi(x_0) > 0$.} there exists $h > 0$, such that
\begin{equation}\label{eq:F_pos}
F\big(x, \var(x), \var_x(x), \var(\cdot) \big) > 0,
\end{equation}
for all $x \in B_h(x_0) := (x_0 - h, x_0 + h)$, such that
\begin{equation}\label{eq:max_strict}
v_{+}(x) - \var(x) < 0, \quad \text{for all} ~ x \in B_h(x_0) \backslash \{x_0 \},
\end{equation}
and such that
\begin{equation}\label{eq:var_pos}
\var(x) > 0, \quad \text{for all} ~ x \in B_h(x_0).
\end{equation}

Because the set $B := \overline{B_h(x_0)} \backslash B_{h/2}(x_0)$ is compact, because $v_{+} - \var$ is u.s.c., and because of inequality \eqref{eq:max_strict} on $B_h(x_0) \backslash \{x_0\}$, there exists a $\del > 0$ such that
$$
v_{+}(x) + \del \le \var(x), \quad \hbox{for all} ~ x \in B.
$$
By Proposition 4.1 in Bayraktar and S\^irbu \cite{BS2012} and by Lemma \ref{lem:sup-comb}, we know that $v_{+}$ is the limit of a non-increasing sequence of stochastic supersolutions $\{ v_n \}_{n \in \N}$.  Fix $\del' \in (0, \del)$, and define a sequence of sets $\{A_n\}_{n \in \N}$ by
\[
A_n = \left\{x \in B: v_n(x) + \del' \ge \var(x)\right\}.
\]
Because each $v_n - \var$ is u.s.c., each $A_n$ is closed.  Also, because $\{v_n\}_{n \in \N}$ is non-increasing, it follows that $\{ A_n \}_{n\in\N}$ is non-increasing; hence, $\bigcap_{n \in \N} A_n = \emptyset$ because $\del' < \del$.  Therefore, because each $A_n$ is closed, there exists $N$ such that $A_n = \emptyset$ for all $n \ge N$, which implies
\begin{equation*}
v_n(x) + \del' < \var(x), \quad \text{for all} ~ n \ge N ~ \text{and} ~ x \in B.
\end{equation*}
For $\veps > 0$, define $\var^\veps$ on $\R$ by $\var^\veps(x) = \var(x) - \veps$.  Note that
\begin{align*}
F\big(x, \var^\veps(x), \var^\veps_x(x), \var^\veps(\cdot) \big) &= F\big(x, \var(x), \var_x(x), \var(\cdot) \big) + \bet (\var^\veps(x) - \var(x)) + \rho( \var^\veps(x) - \var(x))  {\bf 1}_{\{x < 0\}} \\
&=  F\big(x, \var(x), \var_x(x), \var(\cdot) \big) - \veps \big( \bet + \rho {\bf 1}_{\{x < 0\}} \big),
\end{align*}
for all $x \in \R$.  Then, \eqref{eq:F_pos} and \eqref{eq:var_pos} imply that there exists a $\veps' \in (0, \del')$ small enough so that
\begin{align}
\label{oper:F_ve}
F\big(x, \var^{\veps'}(x), \var^{\veps'}_x(x), \var^{\veps'}(\cdot) \big) > 0, \quad \text{for all} ~ x \in B_h(x_0),
\end{align}
and
\begin{equation}\label{eq:vareps_pos}
\var^{\veps'}(x) \ge 0, \quad \text{for all} ~ x \in B_h(x_0).
\end{equation}
We also have, for all $n \ge N$ and $x \in B$,
\begin{align}\label{eq:vn_varveps}
v_n(x) < \var(x) - \del' < \var(x) - \veps' = \var^{\veps'}(x),
\end{align}
and
\begin{align}
\var^{\veps'}(x_0) < \var(x_0) = v_{+}(x_0) \le v_n(x_0).
\label{ineq:com-var}
\end{align}
For $n \ge N$, define the function $\iota^{\veps'}_n$ on $\R$ by
\begin{equation}\label{eq:iota}
\iota^{\veps'}_n(x) = \chi(x) \var^{\veps'}(x) + \big(1 - \chi(x) \big) v_n(x),
\end{equation}
in which $\chi$ is a continuously differentiable function satisfying
\begin{equation*}
\begin{cases}
 0 \le \chi (x) \le 1,  \qquad &\text{for all} ~ x \in \R, \\
 \chi(x) = 1, \qquad  &\text{for} ~ x \in \overline{B_{h/2}(x_0)},  \\
 \chi(x) = 0, \qquad  &\text{for} ~ x \in \overline{B_h(x_0)}^{\, c}.
\end{cases}
\end{equation*}

Next, fix some $n \ge N$, and define the function $v^{\veps'}$ on $\R$ by
\begin{equation}
\label{eq:v_eps}
v^{\veps'}(x) = v_n(x) \wedge \iota^{\veps'}_n(x).
\end{equation}
We claim $v^{\veps'} = v_n$ outside ${B_{h/2}(x_0)}$.  Indeed, because $\chi(x) = 0$ for $x \in \overline{B_h(x_0)}^{\, c}$, we have $v^{\veps'} = v_n$ outside $\overline{B_h(x_0)}$.  Also, on $B = \overline{B_h(x_0)} \backslash {B_{h/2}(x_0)}$, inequality \eqref{eq:vn_varveps} gives us $v_n < \var^{\veps'}$, which implies $v_n \le \iota^{\veps'}_n$; thus, $v^{\veps'} = v_n$ on $B$.  Moreover, $\iota^{\veps'}_n =  \var^{\veps'}$ on $B_{h/2}(x_0)$.  To summarize this discussion, we can express $v^{\veps'}$ as follows:
\begin{equation}
\label{eq:v_eps2}
v^{\veps'}(x) =
\begin{cases}
v_n(x) \wedge \var^{\veps'}(x), &\quad x \in B_{h/2}(x_0), \\
v_n(x), &\quad x \notin B_{h/2}(x_0).
\end{cases}
\end{equation}
In light of \eqref{ineq:com-var}, \eqref{eq:iota}, \eqref{eq:v_eps}, and \eqref{eq:v_eps2} we  have
$$
v^{\veps'}(x_0) = \iota^{\veps'}_n(x_0) = \var^{\veps'}(x_0) < v_{+}(x_0).
$$
Thus, if we can show $v^{\veps'} \in \Psi^+$, then we will contradict the pointwise minimality of $v_{+}$, and the proof will be complete.

To prove $v^{\veps'} \in \Psi^+$, note that $v^{\veps'}$ is u.s.c.\ because $v_n$ and $\iota^{\veps'}_n$ are both u.s.c.  The function $v^{\veps'}$ takes values in $\left[ 0, \, \frac{\rho}{\rho + \bet} \right]$ because $v_n$ takes values in that interval and $\var^{\veps'}$ is non-negative on $B_{h/2}(x_0)$.  Condition (2) in Definition \ref{def:stoch_super} is satisfied because $v^{\veps'} = v_n$ outside ${B_{h/2}(x_0)}$. Therefore, we only need to verify condition (1), that is, the supermartingale property.

Let $(\tau, \zeta)$ be any random initial condition and $\mR_0 \in \fR_{\tau, K}$ be the $(\tau, \zeta)$-admissible control in condition (1) associated with the stochastic supersolution $v_n$.  Let $A$ denote the event
\begin{equation}\label{eq:A_v}
A = \big\{\zeta \in B_{h/2}(x_0) ~ \hbox{and} ~ \iota^{\veps'}_n(\zeta) < v_n(\zeta) \big\}.
\end{equation}
Because $\iota^{\veps'}_n = \var^{\veps'}$ on $\overline{B_{h/2}(x_0)}$, we see from \eqref{oper:F_ve} that
\begin{equation}\label{eq:F_v_veps}
F\big(x, \iota^{\veps'}_n(x), \iota^{\veps'}_{n,x}(x), \iota^{\veps'}_n(\cdot)\big) > 0, \quad \text{for all} ~ x \in \overline{B_{h/2}(x_0)}.
\end{equation}
Hence, for each $x \in \overline{B_{h/2}(x_0)}$, there exists a retention function $R_\iota = R_\iota(x, y)$ such that
\begin{align}
&\bet \iota^{\veps'}_n + \rho \big( \iota^{\veps'}_n - 1 \big) {\bf 1}_{\{x < 0\}} + \left( \kap -  \la \left( (1 + \tet) \E R_{\iota} + \eta \E(YR_{\iota}) - \dfrac{\eta}{2} \, \E \big(R_{\iota}^2 \big) \right) \right) \iota^{\veps'}_{n,x} \notag \\
&\quad - \la \left( \E \iota^{\veps'}_n(x - R_{\iota}) - \iota^{\veps'}_n(x) \right) > 0.
\label{ieq:iot_v}
\end{align}
Define a new admissible control $\mR_1 \in \fR_{\tau, K}$ by
\begin{equation}\label{eq:R1}
\mR_1 = \mR_\iota {\bf 1}_A + \mR_0 {\bf 1}_{A^c},
\end{equation}
define the stopping time
\[
\tau_1 = \left\{t \ge \tau: X^{\tau,\zeta, \mR_1}_t \notin B_{h/2}(x_0) \right\},
\]
and define the random variable
\[
\zeta_1 = X^{\tau,\zeta, \mR_1}_{\tau_1}.
\]
In \eqref{eq:R1}, the strategy $\mR_\iota$ is defined on $A$ via $R_{\iota, t} = R_\iota(X_{t-}, y)$ for $X_{t-} \in A$, in which $R_\iota(x, y)$ on the right side is the function that satisfies inequality \eqref{ieq:iot_v}. Note that, because $X^{\tau, \zeta, \mR_1}$ follows a jump process, $\zeta_1$ might not lie on the boundary of $B_{h/2}(x_0)$.

By applying a general version of It\^o's formula (see Protter \cite{P2005}) to $e^{-\bet(\tau_1 \wedge \omega)} \iota^{\veps'}_n \Big( X^{\tau,\zeta, \mR_1}_{\tau_1 \wedge \omega} \Big)$, we have\footnote{Note that $\iota^{\veps'}_n = \var^{\veps'}$ on $B_{h/2}(x_0)$ and is, thus, continuously differentiable on that interval.}
\begin{align}
&e^{-\bet(\tau_1 \wedge \omega)} \iota^{\veps'}_n \Big( X^{\tau,\zeta, \mR_1}_{\tau_1 \wedge \omega} \Big) \notag\\
&= e^{-\bet \tau} \iota^{\veps'}_n(\zeta)
 - \int^{\tau_1 \wedge \omega}_{\tau} e^{-\bet t} \bigg[ \bet \iota^{\veps'}_n \big(X^{\tau,\zeta, \mR_1}_t \big) + \rho \big(\iota^{\veps'}_n \big(X^{\tau,\zeta, \mR_1}_t \big) - 1 \big) {\bf 1}_{\{X^{\tau,\zeta, \mR_1}_t < 0\}} \notag \\
&\; \qquad \qquad \qquad \qquad \qquad \qquad + \left( \kap - \la \left( (1 + \tet) \E {R_1} + \eta \E(Y{R_1}) - \dfrac{\eta}{2} \, \E \big({R^2_1} \big) \right) \right) \iota^{\veps'}_{n,x} \big(X^{\tau,\zeta, \mR_1}_t \big) \notag \\
&\; \qquad \qquad \qquad \qquad \qquad \qquad  - \la \left( \E \iota^{\veps'}_n \big(X^{\tau,\zeta, \mR_1}_t - R_{1,t} \big) - \iota^{\veps'}_n(X^{\tau,\zeta, \mR_1}_t) \right) \bigg] dt  \notag\\
&\quad + M_{\tau_1 \wedge \omega} + \int^{\tau_1 \wedge \omega}_{\tau} e^{-\bet t} \big(1 - \iota^{\veps'}_n \big(X^{\tau,\zeta, \mR_1}_t \big) \big) {\bf 1}_{\{X^{\tau,\zeta, \mR_1}_t < 0\}} d(N_t^{\rho} - \rho t),
\label{Ito_iot}
\end{align}
in which $N^{\rho}$ is the jump process associated with exponential Parisian ruin, and
\begin{align}
M_t &= \sum \limits_{\substack{ X^{\tau,\zeta, \mR_1}_s \ne X^{\tau,\zeta, \mR_1}_{s-} \\ \tau \le s \le t }} e^{-\bet s} \left( \iota^{\veps'}_n \big(X^{\tau,\zeta, \mR_1}_s \big) - \iota^{\veps'}_n \big(X^{\tau,\zeta, \mR_1}_{s-} \big) \right) \notag \\
& \quad - \la \int^{t}_{\tau} e^{-\bet s} \left(\E \iota^{\veps'}_n \big(X^{\tau,\zeta, \mR_1}_s - R_{1,s} \big) - \iota^{\veps'}_n \big(X^{\tau,\zeta, \mR_1}_s \big) \right) ds
\label{eq:M_t}
\end{align}
is a martingale with zero $\mF_\tau$-expectation.  From \eqref{ieq:iot_v}, we know that the integrand of the first integral is non-negative.  Because $\iota^{\veps'}_n$ is bounded on $B_{h/2}(x_0)$, the $\mF_\tau$-expectation of the second integral equals zero.  Thus, by taking the $\mF_\tau$-expectation of the expression in \eqref{Ito_iot}, we obtain
\[
e^{-\bet \tau} \iota^{\veps'}_n(\zeta) {\bf 1}_A \ge \E \left[ e^{-\bet(\tau_1 \wedge \omega)} \iota^{\veps'}_n \Big( X^{\tau,\zeta, \mR_1}_{\tau_1 \wedge \omega} \Big) {\bf 1}_A \bigg| \mF_{\tau} \right]
~~a.s.\]
By the definition of $v^{\veps'}$, we also have
\begin{align*}
e^{-\bet(\tau_1 \wedge \omega)} \iota^{\veps'}_n \Big( X^{\tau,\zeta, \mR_1}_{\tau_1 \wedge \omega} \Big) {\bf 1}_A
\ge e^{-\bet(\tau_1 \wedge \omega)} v^{\veps'} \Big( X^{\tau,\zeta, \mR_1}_{\tau_1 \wedge \omega} \Big) {\bf 1}_A.
\end{align*}
Thus, from the definition of $A$ in \eqref{eq:A_v} and from the above inequalities, we obtain
\begin{align}
\label{ineq:vet_A}
e^{-\bet \tau} v^{\veps'}(\zeta) {\bf 1}_A = e^{-\bet \tau} \iota^{\veps'}_n(\zeta) {\bf 1}_A
\ge \E\left[ e^{-\bet(\tau_1 \wedge \omega)} v^{\veps'} \Big( X^{\tau,\zeta, \mR_1}_{\tau_1 \wedge \omega} \Big) {\bf 1}_A  \bigg| \mF_{\tau}\right]~~a.s.
\end{align}
On the other hand, we know that $\mR_0$ is the $(\tau, \zeta)$-control associated with $v_n$, $\mR_1 = \mR_0$ on $A^c$, and
$v^{\veps'} = v_n$ outside $B_{h/2}(x_0)$; thus, from the supermartingale property of $v_n$, we have
\begin{align}
\label{ineq:vet_Ac}
e^{-\bet \tau} v^{\veps'}(\zeta) {\bf 1}_{A^c} = e^{-\bet \tau} v_n(\zeta) {\bf 1}_{A^c}
&\ge \E \left[ e^{-\bet(\tau_1 \wedge \omega)} v_n \Big( X^{\tau,\zeta, \mR_1}_{\tau_1 \wedge \omega} \Big) {\bf 1}_{A^c} \bigg|\mF_{\tau} \right] \notag \\
&\ge \E\left[ e^{-\bet(\tau_1 \wedge \omega)} v^{\veps'} \Big( X^{\tau,\zeta, \mR_1}_{\tau_1 \wedge \omega} \Big) {\bf 1}_{A^c}  \bigg|\mF_{\tau} \right]~~a.s.
\end{align}
By combining \eqref{ineq:vet_A} and \eqref{ineq:vet_Ac}, we get
\begin{align}
e^{-\bet \tau} v^{\veps'}(\zeta)\ge \E\left[ e^{-\bet(\tau_1 \wedge \omega)} v^{\veps'} \Big( X^{\tau,\zeta, \mR_1}_{\tau_1 \wedge \omega} \Big) \bigg| \mF_{\tau} \right]~~a.s.
\label{ineq:vet}
\end{align}

Let $\mR_2 \in \fR_{\tau, K}$ denote the control associated with $v_n$ for the starting time $\tau_1$ and initial condition $\zeta_1$.  Define a new retention strategy $\mR' \in \fR_{\tau, K}$ by
\[
\mR' = \mR_1 {\bf 1}_{\{\tau < t \le \tau_1\}} + \mR_2 {\bf 1}_{\{t > \tau_1\}}.
\]
Rewrite the right side of inequality \eqref{ineq:vet} as
\begin{align}
\label{eq:vet_div}
\E\left[ e^{-\bet(\tau_1 \wedge \omega)} v^{\veps'} \Big( X^{\tau,\zeta, \mR_1}_{\tau_1 \wedge \omega} \Big) \bigg| \mF_{\tau} \right]
= \E\left[ e^{-\bet \omega} v^{\veps'} \Big(X^{\tau,\zeta, \mR'}_{\omega} \Big) {\bf 1}_{\{\omega \le \tau_1\}}
+ e^{-\bet \tau_1} v^{\veps'}(\zeta_1) {\bf 1}_{\{\omega > \tau_1\}}  \bigg| \mF_{\tau} \right]~~a.s.
\end{align}
Also, because $v^{\veps'} = v_n$ outside $B_{h/2}(x_0)$, from the supermartingale property of $v_n$, we obtain
\begin{align}
\label{ineq:rho_tau}
e^{-\bet \tau_1} v^{\veps'}(\zeta_1) {\bf 1}_{\{\omega > \tau_1\}} = e^{-\bet \tau_1} v_n(\zeta_1) {\bf 1}_{\{\omega > \tau_1\}}
& \ge \E\left[ e^{-\bet \omega} v_n \Big( X^{\tau_1,\zeta_1, \mR_2}_{\omega} \Big) {\bf 1}_{\{\omega > \tau_1\}} \bigg| \mF_{\tau_1} \right] \notag\\
& \ge \E\left[ e^{-\bet \omega} v^{\veps'} \Big( X^{\tau_1,\zeta_1, \mR'}_{\omega} \Big) {\bf 1}_{\{\omega > \tau_1\}} \bigg| \mF_{\tau_1} \right]~~ a.s.
\end{align}
By substituting \eqref{ineq:rho_tau} into \eqref{eq:vet_div}, and by combining the result with \eqref{ineq:vet}, we finally get
\begin{align*}
e^{-\bet \tau} v^{\veps'}(\zeta)\ge \E\left[ e^{-\bet \omega} v^{\veps'} \Big( X^{\tau,\zeta, \mR'}_{\omega} \Big) \bigg| \mF_{\tau}\right]~~a.s.
\end{align*}
Hence, $v^{\veps'} \in \Psi^+$, with associated $(\tau, \zeta)$-retention strategy $\mR'$, and we have completed the proof of this theorem.
\end{proof}

An immediate corollary of the definition of $v_{+}$ in \eqref{eq:v_plus} and of Lemma \ref{lem:psi_v} is the following.

\begin{cor}\label{cor:psi_vplus}
$\psi \le v_{+}$ on $\R$, that is, the minimum discounted probability of exponential Parisian ruin is a lower bound of $v_{+}$.  \qed
\end{cor}

\subsection{Stochastic subsolution}\label{sec:42}

\begin{defin} \label{def:stoch_sub}
An l.s.c.\ function $u: \R \to \left[ 0, \, \frac{\rho}{\rho + \bet} \right]$ is called a {\rm stochastic subsolution} if it satisfies the following properties:
\begin{itemize}
\item[$(1)$]
For any random initial condition $(\tau, \zeta)$, any retention strategy $\mR \in \fR_{\tau, K}$, and any $\mathbb{F}$-stopping time $\omega \in [\tau, K]$,
\[
e^{-\bet \tau} u(\zeta) \le \E \left[ e^{-\bet \omega} u \big({X}^{\tau, \zeta, \mR}_\omega \big) \Big | \, \mF_{\tau} \right] ~~ a.s.,
\]
in which $u$ is understood to be its extension to $\R \cup \{\Delta \}$.
\item[$(2)$] $\lim \limits _{x \to - \infty} u(x) \le \dfrac{\rho}{\rho + \bet}$ and $\lim \limits _{x \to \infty} u(x) = 0$.
\end{itemize}
Let ${\Psi^-}$ denote the set of stochastic subsolutions.  \qed
\end{defin}

$\Psi^-$ is non-empty because $0 \in \Psi^-$.  However, it is more useful to have a stochastic subsolution that satisfies the boundary conditions with equality; therefore, we present the following lemma.

\begin{lemma}\label{lem:upsi}
The function $\upsi$ is a stochastic subsolution, in which $\upsi$ is defined by
\begin{equation}\label{eq:upsi}
\upsi(x) =
\begin{cases}
\dfrac{\rho}{\rho + \bet} \left( 1 - e^{\gamt(x +1)} \right), &\quad x < -1, \\
0, &\quad x \ge -1,
\end{cases}
\end{equation}
with $\gamt$ equal to the unique positive solution of \eqref{eq:gamt}.
\end{lemma}

\begin{proof}
By construction, $\upsi$ is in $\mC^0(\R)$, and it decreases from $\frac{\rho}{\rho + \bet}$ to $0$ on $\R$; thus, $\upsi$ is l.s.c.\ and satisfies condition (2) in Definition \ref{def:stoch_sub} with equality.  Thus, we only need to show condition (1) of that definition.  First, we see, for $x > -1$,
\[
F\big(x, \,\upsi(x), \,\upsi'(x),\, \upsi(\cdot) \big) \le 0.
\]
For $x < -1$, by the definition of $\gamt$ and the proof of Proposition \ref{prop:gamt}, one can show that
$$
F\left(x,\, \upsi(x), \, \upsi'(x), \, \upsi(\cdot)\right) = 0.
$$
Finally, for $x = -1$, we have
\[
\lim \limits_{x \to (-1)-} \upsi'(x) = \lim \limits_{x \to (-1)-} \dfrac{\upsi(-1) - \upsi(x)}{-1 - x} = - \, \dfrac{\rho}{\rho + \bet} \, \gam_2,
\]
from which it follows
$$
\lim \limits_{x \to (-1)-} F\big(x,\, \upsi(x), \, \upsi'(x), \, \upsi(\cdot) \big) = 0.
$$
Hence, by applying It\^o's formula (see Theorem 47 in Chapter IV of Protter \cite{P2005}, which uses left limits) to $e^{-\bet \omega} \upsi \big(X_{\omega}^{\tau, \zeta, \mR} \big)$ with initial random condition $(\tau, \zeta)$, retention strategy $\mR \in \fR_{\tau, K}$, and $\mathbb{F}$-stopping time $\omega \in [\tau, K]$, and by taking the $\mF_\tau$-expectation as in the proof of Lemma \ref{lem:opsi}, we obtain
$$
e^{-\bet \tau} \upsi(\zeta) \le \E \left[ e^{-\bet \omega} \upsi \big({X}^{\tau,\zeta, \mR}_\omega \big) \Big | \mF_{\tau} \right] ~~a.s.
$$
Therefore, $ \upsi$ satisfies condition (1) in Definition \ref{def:stoch_sub}.
\end{proof}

\begin{lemma}\label{lem:u_psi}
For any $u \in \Psi^-$, we have $u \le \psi$ on $\R$, that is, the minimum discounted probability of exponential Parisian ruin is an upper bound of any stochastic subsolution.
\end{lemma}

\begin{proof}
First, note that $u \le \psi$ on the boundary of $\R$ by condition (2) in Definition \ref{def:stoch_sub}.  Second, for $x \in \R$, let $(\tau, \zeta) = (0, x)$, let $\mR$ be any admissible retention strategy, and let $\omega = \min(K, \tau_M)$, in which $\tau_M$ is defined by
$$
\tau_M = \inf \left\{t \ge 0 : X^{x, \mR}_t = M \right\}, \quad \text{for some}~ M \ge x,
$$
that is, $\tau_M$ is the first time $X^{x, \mR}$ hits the barrier $M$.  Then, by applying the submartingale property (1) in Definition \ref{def:stoch_sub}, we have
\begin{align*}
u(x) \le \E \Big[ e^{-\bet (K \wedge \tau_M)} u \big(X_{K \wedge \tau_M}^{x, \mR}\big) \Big] = \E \Big[ e^{-\bet K} u \big(X_{K}^{x, \mR} \big) {\bf 1}_{\{K < \tau_M\}} \Big] + \E^x \Big[ e^{-\bet \tau_M} u(M) {\bf 1}_{\{K \ge \tau_M\}} \Big].
\end{align*}
Because $u$ is bounded on $\R$ with $\lim \limits_{M \to \infty} u(M) = 0$, the Dominated Convergence Theorem implies
\[
\lim_{M \to \infty}  \E^x \Big[ e^{-\bet \tau_M} u(M) {\bf 1}_{\{K \ge \tau_M\}} \Big] = \E^x \Big[ \lim_{M \to \infty} e^{-\bet \tau_M} u(M) {\bf 1}_{\{K \ge \tau_M\}} \Big] = 0.
\]
Thus, because $u \big(X_{K}^{x, \mR} \big) {\bf 1}_{\{K < \tau_M\}} = {\bf 1}_{\{K < \tau_M\}}$,
\begin{align}
\label{eq:KM}
u(x) \le \lim_{M \to \infty} \E^{x} \big[ e^{- \bet K} {\bf 1}_{\{K < \tau_M\}} \big].
\end{align}
If we were to prove
\begin{align}
\label{eq:tau_M}
\lim \limits_{M \to \infty} \tau_M\big(X^{x, \mR} \big) = \infty ~ a.s.,
\end{align}
then \eqref{eq:KM} and \eqref{eq:tau_M} would imply, by the Dominated Convergence Theorem,
\begin{align*}
u(x) \le \E^{x} \Big[ \lim_{M \to \infty} e^{- \bet K} {\bf 1}_{\{K < \tau_M\}} \Big] = \E^x \Big[ e^{- \bet K} {\bf 1}_{\{ K < \infty \}} \Big].
\end{align*}
Because this inequality holds for any retention strategy, by taking the infimum over all admissible retention strategies, we obtain $u \le \psi$.

It remains for us to prove \eqref{eq:tau_M}.  To that end, construct a simple process $\hat{X}_t = x + c t$ for $t \ge 0$.  It is easy to verify that $\hat{X}_t \ge X^{x, \mR}_t$ a.s.\ for all $t \ge 0$ and all $\mR \in \fR$. Thus, the hitting time to barrier $M \ge x$ for $\hat{X}$ is less than or equal to the hitting time for $X^{x, \mR}$, that is, $\tau_M \big(\hat{X} \big) \le \tau_M\big(X^{x, \mR}\big)$ a.s.  Moreover, we compute $\tau_M(\hat{X}) = (M - x)/c$, so
$$
\tau_M \big(X^{x, \mR} \big) \ge \frac{M - x}{c} \, ~ a.s.
$$
Therefore, $\lim \limits_{M \to \infty} \tau_M\big(X^{x, \mR} \big) = \infty$ a.s., and we have proved \eqref{eq:tau_M}.
\end{proof}

Similar to Lemma \ref{lem:sup-comb}, we have the following lemma for stochastic subsolutions.

\begin{lemma}\label{lem:sub-comb}
If $u_1$ and $u_2$ are two stochastic subsolutions, then $u = u_1 \vee u_2$ is also a stochastic subsolution.  \qed
\end{lemma}

\begin{thm}\label{thm:u_minus}
The lower stochastic envelope $u_{-},$ defined by
\begin{equation}\label{eq:u_minus}
u_{-}(x) = \sup \limits_{u \in {\Psi^-}} u(x),
\end{equation}
is a viscosity supersolution of \eqref{HJB_F} and \eqref{boundary_condition}.
\end{thm}

\begin{proof}
$u_{-}$ is bounded and l.s.c.\ because it is the pointwise supremum of a set of l.s.c.\ functions that are bounded above the minimum discounted probability of exponential Parisian ruin $\psi$.  $u_{-}$ satisfies the boundary conditions in \eqref{boundary_condition} as $x$ goes to $-\infty$, namely, $\lim \limits_{x \to - \infty} u_{-}(x) = \frac{\rho}{\rho + \bet}$ and $\lim \limits_{x \to \infty} u_{-}(x) = 0$.  Indeed, $\lim \limits_{x \to - \infty} u_{-}(x) \le \frac{\rho}{\rho + \bet}$ from the definition of stochastic subsolution and from the definition of $u_{-}$ in \eqref{eq:u_minus}, $\lim \limits_{x \to -\infty} \upsi(x) = \frac{\rho}{\rho + \bet}$ from the definition of $\upsi$ in \eqref{eq:upsi}, and $u_{-} \ge \upsi$ from Lemma \ref{lem:upsi}.  The second limit, namely, $\lim \limits_{x \to \infty} u_{-}(x) = 0$, follows from the definition of stochastic subsolution.

Next, we show the interior viscosity supersolution property.  Let  $x_0 \in \R$ and $\phi \in \mC^1(\R)$ be such that $u_{-} - \phi$ attains a strict, global minimum at $x_0$ with $u_{-}(x_0) = \phi(x_0)$.  We need to show that
\begin{equation*}
F\big(x_0, \phi(x_0), \phi_x(x_0), \phi(\cdot)\big) \ge 0.
\end{equation*}
Assume, on the contrary, that
\begin{equation*}
F\big(x_0, \phi(x_0), \phi_x(x_0), \phi(\cdot)\big) < 0.
\end{equation*}
Then, by the continuity of $\phi$ and $F$ (away from 0), by the strict minimization of $u_{-} - \phi$ at $x = x_0$, and by $\phi(x_0) = u_{-}(x_0) \le \psi(x_0) \le v_{+}(x_0) \le \opsi(x_0) < \frac{\rho}{\rho + \bet} \,$, there exists $h > 0$ such that
\begin{equation}\label{eq:F_neg}
F\big(x, \phi(x), \phi_x(x), \phi(\cdot)\big) < 0,
\end{equation}
for all $x \in B_h(x_0) := (x_0 - h, x_0 + h)$, and such that
\begin{equation}\label{eq:min_strict}
u_{-}(x) - \phi(x) > 0, \quad \text{for all} ~ x \in B_h(x_0) \backslash \{x_0\},
\end{equation}
and such that
\begin{equation}\label{eq:phi_bnd}
\phi(x) < \dfrac{\rho}{\rho + \bet} \, , \quad \text{for all} ~ x \in B_h(x_0).
\end{equation}

Because the set $B := \overline{B_h(x_0)} \backslash B_{h/2}(x_0)$ is compact, because $u_{-} - \phi$ is l.s.c., and because of inequality \eqref{eq:min_strict} on $B_h(x_0) \backslash \{x_0\}$, there exists a $\del > 0$ such that
$$
u_{-} - \phi \ge \del, \qquad \hbox{for all} ~ x \in B.
$$
Similarly, as in the proof of Theorem \ref{thm:v_plus}, we deduce that $u_{-}$ is the limit a non-decreasing sequence of stochastic subsolutions $\{ u_n\}_{n \in \N}$.  Fix $\del' \in (0, \del)$, and define a sequence of sets $\{A_n\}_{n \in \N}$ by
\[
A_n = \big\{ x \in B: u_n(x) - \phi(x) \le \del' \big\}.
\]
Because each $u_n - \phi$ is l.s.c., each $A_n$ is closed.  Also, because $\{u_n\}_{n \in \N}$ is non-decreasing, it follows that $\{A_n\}_{n \in \N}$ is non-increasing; hence, $\bigcap_{n \in \N} A_n = \emptyset$ because $\del' < \del$.  Therefore, because each $A_n$ is closed, there exists $N$ such that $A_n = \emptyset$ for all $n \ge N$, which implies
\[
u_n(x) > \phi(x) + \del', \quad \hbox{for all} ~ n \ge N ~ \hbox{and} ~ x \in B.
\]
For $\veps > 0$, define $\phi^\veps$ on $\R$ by $\phi^\veps(x) = \phi(x) + \veps$.  Note that
\begin{align*}
F\big(x, \phi^\veps(x), \phi^\veps_x(x), \phi^\veps(\cdot) \big) &= F\big(x, \phi(x), \phi_x(x), \phi(\cdot) \big) + \bet (\phi^\veps(x) - \phi(x)) + \rho( \phi^\veps(x) - \phi(x))  {\bf 1}_{\{x < 0\}}\\
&=  F\big(x, \phi(x), \phi_x(x), \phi(\cdot) \big) + \veps \big( \bet + \rho  {\bf 1}_{\{x < 0\}} \big),
\end{align*}
for all $x \in \R$.  Then, \eqref{eq:F_neg} and \eqref{eq:phi_bnd} imply that there exists a $\veps' \in (0, \del')$ small enough so that
\begin{equation}\label{eq:F_u}
F \big(x, \phi^{\veps'}(x), \phi^{\veps'}_x(x), \phi^{\veps'}(\cdot) \big) < 0, \quad \text{for all} ~ x \in B_h(x_0),
\end{equation}
and
\begin{equation}\label{eq:phieps_pos}
\phi^{\veps'}(x) \le \dfrac{\rho}{\rho + \bet} \, , \quad \text{for all} ~ x \in B_h(x_0).
\end{equation}

We also have, for all $n \ge N$ and $x \in B$,
\begin{align}\label{ineq:u_phi}
u_n(x) > \phi(x) + \del' > \phi(x) + \veps' = \phi^{\veps'}(x),
\end{align}
and
\begin{align}\label{ineq:com-phi}
\phi^{\veps'}(x_0) > \phi(x_0) = u_{-}(x_0) \ge u_n(x_0).
\end{align}
For $n \ge N$, define the function $\iota^{\veps'}_n$ on $\R$ by
\begin{equation}\label{eq:iota_u}
\iota^{\veps'}_n(x) = \chi(x) \phi^{\veps'}(x) + \big(1 - \chi(x) \big) u_n(x),
\end{equation}
in which $\chi$ is a continuously differentiable function satisfying
\begin{equation*}
\begin{cases}
 0 \le \chi (x) \le 1,  \qquad &\text{for all} ~ x \in \R, \\
 \chi(x) = 1, \qquad  &\text{for} ~ x \in \overline{B_{h/2}(x_0)},  \\
 \chi(x) = 0, \qquad  &\text{for} ~ x \in \overline{B_h(x_0)}^{\, c}.
\end{cases}
\end{equation*}

Next, fix some $n \ge N$, and define the function $u^{\veps'}$ on $\R$ by
\begin{equation}
\label{eq:u_eps}
u^{\veps'}(x) = u_n(x) \vee \iota^{\veps'}_n(x).
\end{equation}
We claim $u^{\veps'} = u_n$ outside ${B_{h/2}(x_0)}$.  Indeed, because $\chi(x) = 0$ for $x \in \overline{B_h(x_0)}^{\, c}$, we have $u^{\veps'} = u_n$ outside $\overline{B_h(x_0)}$.  Also, on $B = \overline{B_h(x_0)} \backslash {B_{h/2}(x_0)}$, inequality \eqref{ineq:u_phi} gives us $u_n > \phi^{\veps'}$, which implies $u_n \ge \iota^{\veps'}_n$; thus, $u^{\veps'} = u_n$ on $B$.  Moreover, $\iota^{\veps'}_n = \phi^{\veps'}$ on $B_{h/2}(x_0)$. To summarize this discussion, we can expression $u^{\veps'}$ as follows:
\begin{equation}\label{eq:u_eps2}
u^{\veps'} =
\begin{cases}
u_n(x) \vee \phi^{\veps'}(x), &\quad x \in B_{h/2}(x_0), \\
u_n(x), &\quad x \notin B_{h/2}(x_0).
\end{cases}
\end{equation}
In light of \eqref{ineq:com-phi}, \eqref{eq:iota_u}, \eqref{eq:u_eps}, and \eqref{eq:u_eps2}, we  have
$$
u^{\veps'}(x_0) = \iota^{\veps'}_n(x_0) = \phi^{\veps'}(x_0) > u_{-}(x_0).
$$
Thus, if we can show $u^{\veps'} \in \Psi^-$, then we will contradict the pointwise maximality of $u_{-}$, and the proof will be complete.

To prove $u^{\veps'} \in \Psi^-$, note that $u^{\veps'}$ is l.s.c.\ because $u_n$ and $\iota^{\veps'}_n$ are both l.s.c.  The function $u^{\veps'}$ takes values in $\left[ 0, \, \frac{\rho}{\rho + \bet} \right]$ because $u_n$ takes values in that interval and $\phi^{\veps'}$ is less than or equal to $\frac{\rho}{\rho + \bet}$ on $B_{h/2}(x_0)$.  Condition (2) in Definition \ref{def:stoch_super} is satisfied because $u^{\veps'} = u_n$ outside ${B_{h/2}(x_0)}$. Therefore, we only need to verify condition (1), that is, the submartingale property.

Let $(\tau, \zeta)$ be any random initial condition, let $\mR$ be any retention strategy in $\fR_{\tau, K}$, and let $\omega$ be any $\mathbb{F}$-stopping time in $[\tau, K]$.  Let $A$ denote the event
\begin{equation}\label{eq:A_u}
A = \big\{\zeta \in B_{h/2}(x_0) ~ \hbox{and} ~ \iota^{\veps'}_n(\zeta) > u_n(\zeta) \big\}.
\end{equation}
Because $\iota^{\veps'}_n = \phi^{\veps'}$ on $\overline{B_{h/2}(x_0)}$, we see from \eqref{eq:F_u} that
\begin{equation}\label{eq:F_u_iota}
F\big(x, \iota^{\veps'}_n(x), \iota^{\veps'}_{n,x}(x), \iota^{\veps'}_n(\cdot)\big) < 0, \quad \text{for all} ~ x \in \overline{B_{h/2}(x_0)}.
\end{equation}
Hence,  for any time-$t$ realization $R_t = R$ of the retention strategy,
\begin{align}
&\bet \iota^{\veps'}_n + \rho \big( \iota^{\veps'}_n - 1 \big) {\bf 1}_{\{x < 0\}} + \left( \kap -  \la \left( (1 + \tet) \E R + \eta \E(YR) - \dfrac{\eta}{2} \, \E \big(R^2 \big) \right) \right) \iota^{\veps'}_{n,x} \notag \\
&\quad - \la \left( \E \iota^{\veps'}_n(x - R) - \iota^{\veps'}_n(x) \right) < 0,
\label{ieq:iot_u}
\end{align}
for all $x \in \overline{B_{h/2}(x_0)}$. Define the stopping time
\[
\tau_1 = \left\{t \ge \tau: X^{\tau,\zeta, \mR}_t \notin B_{h/2}(x_0) \right\},
\]
and define the random variable
\[
\zeta_1 = X^{\tau,\zeta, \mR}_{\tau_1}.
\]
Because $X^{\tau, \zeta, \mR}$ follows a jump process, $\zeta_1$ might not lie on the boundary of $B_{h/2}(x_0)$.

By applying a general version of It\^o's formula to $e^{-\bet (\tau_1 \wedge \omega)} \iota^{\veps'}_n \Big( X^{\tau,\zeta, \mR}_{\tau_1 \wedge \omega} \Big)$, we obtain \eqref{Ito_iot} with $\mR_1$ replaced by $\mR$.  From \eqref{ieq:iot_u}, we know that the integrand of the first integral in the analog of \eqref{Ito_iot} is non-positive.  As before, because $\iota^{\veps'}_n$ is bounded on $B_{h/2}(x_0)$, the $\mF_\tau$-expectation of the second integral equals zero.  Thus, by taking the $\mF_\tau$-expectation of the analog of \eqref{Ito_iot}, we obtain
\[
e^{-\bet \tau} \iota^{\veps'}_n(\zeta) {\bf 1}_A \le \E \left[ e^{-\bet (\tau_1 \wedge \omega)} \iota^{\veps'}_n \Big( X^{\tau,\zeta, \mR}_{\tau_1 \wedge \omega} \Big) {\bf 1}_A \bigg| \mF_{\tau} \right]
~~a.s.\]
By the definition of $u^{\veps'}$, we also have
\begin{align*}
e^{-\bet (\tau_1 \wedge \omega)} \iota^{\veps'}_n \Big( X^{\tau,\zeta, \mR}_{\tau_1 \wedge \omega} \Big) {\bf 1}_A
\le e^{-\bet (\tau_1 \wedge \omega)} u^{\veps'} \Big( X^{\tau,\zeta, \mR}_{\tau_1 \wedge \omega} \Big) {\bf 1}_A.
\end{align*}
Thus, from the definition of $A$ in \eqref{eq:A_u} and from the above inequalities, we obtain
\begin{align}
\label{ineq:u_A}
e^{-\bet \tau} u^{\veps'}(\zeta) {\bf 1}_A = e^{-\bet \tau} \iota^{\veps'}_n(\zeta) {\bf 1}_A
\le \E\left[ e^{-\bet (\tau_1 \wedge \omega)} u^{\veps'} \Big( X^{\tau,\zeta, \mR}_{\tau_1 \wedge \omega} \Big) {\bf 1}_A  \bigg| \mF_{\tau}\right]~~a.s.
\end{align}
On the other hand, $u^{\veps'} = u_n$ on $A^c$; thus, from the submartingale property of $u_n$, we have
\begin{align}
\label{ineq:u_Ac}
e^{-\bet \tau} u^{\veps'}(\zeta) {\bf 1}_{A^c} = e^{-\bet \tau} u_n(\zeta) {\bf 1}_{A^c}
&\le \E \left[ e^{-\bet (\tau_1 \wedge \omega)} u_n \Big( X^{\tau,\zeta, \mR}_{\tau_1 \wedge \omega} \Big) {\bf 1}_{A^c} \bigg|\mF_{\tau} \right] \notag \\
&\le \E\left[ e^{-\bet (\tau_1 \wedge \omega)} u^{\veps'} \Big( X^{\tau,\zeta, \mR}_{\tau_1 \wedge \omega} \Big) {\bf 1}_{A^c}  \bigg|\mF_{\tau} \right]~~a.s.
\end{align}
By combining \eqref{ineq:u_A} and \eqref{ineq:u_Ac}, we get
\begin{align}
e^{-\bet \tau} u^{\veps'}(\zeta)\le \E\left[ e^{-\bet (\tau_1 \wedge \omega)} u^{\veps'} \Big( X^{\tau,\zeta, \mR}_{\tau_1 \wedge \omega} \Big) \bigg| \mF_{\tau} \right]~~a.s.
\label{ineq:u}
\end{align}
Rewrite the right side of inequality \eqref{ineq:u} as
\begin{align}
\label{eq:u_div}
\E\left[ e^{-\bet (\tau_1 \wedge \omega)} u^{\veps'} \Big( X^{\tau,\zeta, \mR}_{\tau_1 \wedge \omega} \Big) \bigg| \mF_{\tau} \right]
= \E\left[ e^{-\bet \omega} u^{\veps'} \Big(X^{\tau,\zeta, \mR}_{\omega} \Big) {\bf 1}_{\{\omega \le \tau_1\}}
+ e^{- \bet \tau_1} u^{\veps'}(\zeta_1) {\bf 1}_{\{\omega > \tau_1\}}  \bigg| \mF_{\tau} \right]~~a.s.
\end{align}
Also, because $u^{\veps'} = u_n$ outside $B_{h/2}(x_0)$, from the submartingale property of $u_n$, we obtain
\begin{align}
\label{ineq:rho_u}
e^{-\bet \tau_1} u^{\veps'}(\zeta_1) {\bf 1}_{\{\omega > \tau_1\}} = e^{-\bet \tau_1} u_n(\zeta_1) {\bf 1}_{\{\omega > \tau_1\}}
& \le \E\left[ e^{-\bet \omega} u_n \Big( X^{\tau_1,\zeta_1, \mR}_{\omega} \Big) {\bf 1}_{\{\omega > \tau_1\}} \bigg| \mF_{\tau_1} \right] \notag\\
& \le \E\left[ e^{-\bet \omega} u^{\veps'} \Big( X^{\tau_1,\zeta_1, \mR}_{\omega} \Big) {\bf 1}_{\{\omega > \tau_1\}} \bigg| \mF_{\tau_1} \right]~~ a.s.
\end{align}
By substituting \eqref{ineq:rho_u} into \eqref{eq:u_div}, and by combining the result with \eqref{ineq:u}, we finally get
\begin{align*}
e^{-\bet \tau} u^{\veps'}(\zeta) \le \E\left[ e^{-\bet \omega} u^{\veps'} \Big( X^{\tau,\zeta, \mR}_{\omega} \Big) \bigg| \mF_{\tau}\right]~~a.s.
\end{align*}
Hence, $u^{\veps'} \in \Psi^-$, and we have completed the proof of this theorem.
\end{proof}

An immediate corollary of the definition of $u_{-}$ in \eqref{eq:u_minus} and of Lemma \ref{lem:u_psi} is the following.

\begin{cor}\label{cor:uminus_psi}
$u_{-} \le \psi$ on $\R$, that is, the minimum discounted probability of exponential Parisian ruin is an upper bound of $u_{-}$.  \qed
\end{cor}

\subsection{Comparison principle}\label{sec:43}

We introduce an equivalent definition of viscosity sub- and supersolutions for our problem; see Appendix \ref{sec:A} for the proof of its equivalence to Definition \ref{def:strict_glob}.

\begin{defin}\label{def:local_vis_val}
We say an upper semi-continuous $($u.s.c.$)$ function $\underline{u}: \R \to \left[ 0, \, \frac{\rho}{\rho + \bet} \right]$ is a {\rm viscosity subsolution} of \eqref{HJB_F} and \eqref{boundary_condition} if \eqref{boundary_condition} holds and if, for any $x_0 \in \R$ and for any $\var \in \mC^1(\R)$ such that $\underline{u} - \var$ reaches a local maximum at $x_0$, we have
\begin{equation}
\begin{cases}
F\big(x_0, \underline{u}(x_0), \var_x(x_0), \underline{u}(\cdot) \big) \le 0, &\quad x_0 \ne 0, \\
\lim \limits_{x \to 0-} F\big(x, \underline{u}(x), \var_x(x), \underline{u}(\cdot) \big) \le 0, &\quad x_0 = 0.
\end{cases}
\end{equation}
Similarly, we say a lower semi-continuous $($l.s.c.$)$ function $\bar{u}: \R \to \left[ 0, \, \frac{\rho}{\rho + \bet} \right]$ is a {\rm viscosity supersolution} of \eqref{HJB_F} and \eqref{boundary_condition} if \eqref{boundary_condition} holds and if, for any $x_0 \in \R$ and for any $\phi \in \mC^1(\R)$ such that $\bar{u} - \phi$ reaches a local minimum at $x_0$, we have
\begin{equation}
F\big(x_0, \bar{u}(x_0), \phi_x(x_0), \bar{u}(\cdot) \big) \ge 0.
\end{equation}
Finally, a function $u$ is called a $($continuous$)$ {\rm viscosity solution} of \eqref{HJB_F} and \eqref{boundary_condition} if it is both a viscosity subsolution and a viscosity supersolution of \eqref{HJB_F} and \eqref{boundary_condition}.  \qed
\end{defin}

We use Definition \ref{def:local_vis_val} to prove a comparison principle.  First, we introduce a function that we will use in that proof.

\begin{lemma}\label{lem:qm}
For a given constant $b > 0$, define the function $q \in \mC^1(\R)$ by
\begin{equation}\label{eq:q}
q(x) =
\begin{cases}
0, &\quad |x| \le 1, \\
b \big(1 + \cos (\pi x) \big), &\quad 1 < |x| < 2, \\
2b, &\quad |x| \ge 2.
\end{cases}
\end{equation}
Then, for $m \in \N$, define the function $\qm$ by
\begin{equation}\label{eq:qm}
\qm(x) = q(x/m),
\end{equation}
for all $x \in \R$.  Then,
\begin{equation}\label{eq:qm_deriv_lim}
\lim_{m \to \infty} || \qm' ||_\infty = 0,
\end{equation}
and
\begin{equation}\label{eq:maxR_qm}
\sup_R \, \E \qm(x - R) - \qm(x) =
\begin{cases}
0, &\quad |x| \ge 2m, \vspace{1ex} \\
- \int_{x+m}^{x + 2m} \qm'(x - y) S_Y(y) dy, &\quad -2m < x \le m, \vspace{1ex} \\
- \int_{2x}^{x + 2m} \qm'(x - y) S_Y(y) dy, &\quad m < x < 2m,
\end{cases}
\end{equation}
which is non-negative for all $x \in \R$.  Furthermore,
\begin{equation}\label{eq:qm_R_lim}
\lim_{m \to \infty} \Big| \Big| \sup_R \, \E \qm(x - R) - \qm(x) \Big| \Big|_\infty = 0.
\end{equation}
\end{lemma}

\noindent{\it Proof.}
The limit in \eqref{eq:qm_deriv_lim} follows easily from
\begin{equation}\label{eq:qm_deriv}
\qm'(x) =
\begin{cases}
0, &\quad |x| \le m, \\
- \, \dfrac{b \pi}{m} \, \sin \Big(\dfrac{\pi x}{m} \Big), &\quad m < |x| < 2m, \\
0, &\quad |x| \ge 2m.
\end{cases}
\end{equation}
To prove \eqref{eq:maxR_qm}, consider the maximization problem over the various domains of $x \in \R$.  First, for $x \le -2m$, $\E \qm(x - R) = 2b = \qm(x)$ for all retention functions $R$.  Second, for $x \ge 2m$, one retention function that maximizes $\qm(x - R)$ is identically $0$, so $\sup_R \qm(x - R) = \qm(x)$.  Third, for $-2m < x \le m$, one retention function that maximizes $\qm(x - R)$ is $Y$, so $\sup_R \qm(x - R) = \E \qm(x - Y)$.  Fourth and finally, for $m < x < 2m$, the optimal $R$ to maximize $\qm(x - R)$ equals $R_m$ given by
\begin{equation}\label{eq:Rm}
R_m(y) =
\begin{cases}
0, &\quad 0 \le y \le 2x, \\
y, &\quad y > 2x.
\end{cases}
\end{equation}
To obtain the expressions in \eqref{eq:maxR_qm} perform integration by parts on $\E \qm(x - Y)$ and $\E \qm(x - R_m)$.

The non-negativity of $\sup_{R} \, \E \qm(x - R) - \qm(x)$ follows from $\qm'(x) \le 0$ for $x \le 0$, and the limit in \eqref{eq:qm_R_lim} follows from
\[
\; \quad \qquad \qquad \qquad \qquad \qquad \qquad \int_{a}^{x + 2m} \big| \qm'(x - y) \big| S_Y(y) dy \le \dfrac{b \pi}{m} \, \E Y.  \qquad \quad \qquad \qquad \qquad \qquad \; \qed
\]

Next, we adapt the proof of Proposition 3.4.1 in Barles and Chasseigne \cite{BC2018} to prove the following lemma.

\begin{lemma}\label{lem:vdel}
If $v$ is a viscosity subsolution of \eqref{HJB_F} and \eqref{boundary_condition}, then
\begin{equation}\label{eq:vdel_lim}
v(x) = \limsup \limits_{\del \to 0+} v(x + \del).
\end{equation}
\end{lemma}

\begin{proof}
Because $v$ is u.s.c., we know that $v(x) \ge \limsup_{\del \to 0} v(x + \del) \ge \limsup_{\del \to 0+} v(x + \del)$ for all $x \in \R$, so suppose that there exists $x_0 \in \R$ such that
\[
v(x_0) > \limsup \limits_{\del \to 0+} v(x_0 + \del) \ge 0.
\]
Then, there exists $\epsilon > 0$ and $\del > 0$ such that $0 < \del' \le \del$ implies
\begin{equation}\label{eq:v_ineq}
v(x_0 + \del') < v(x_0) - \epsilon.
\end{equation}

For $n \in \N$ and for some $\del' \in (0, \del]$, define the function $\Phin$ on $\R$ by
\begin{equation}\label{eq:Phin}
\Phin(x) = v(x) + n(x - x_0) - \dfrac{v(x_0 - \del')}{(\del')^2} \, (x - x_0)^2.
\end{equation}
We wish to maximize $\Phin$ on the interval $[x_0 - \del', x_0 + \del']$.  Because $\Phin$ is u.s.c., it attains its maximum on this interval, say, at $x_n$.

First, for $x \in (x_0, x_0 + \del']$, inequality \eqref{eq:v_ineq} implies
\begin{align*}
\Phin(x) &= v(x) + n(x - x_0) - \dfrac{v(x_0 - \del')}{(\del')^2} \, (x - x_0)^2 \\
&\le v(x) + n(x - x_0) < v(x_0) - \epsilon + n(x - x_0) \le v(x_0) - \epsilon + n \del' \le v(x_0) = \Phin(x_0),
\end{align*}
in which the last inequality follows if we choose $\del'$ such that $n \del' \le \epsilon$.  Thus, on $[x_0 - \del', x_0 + \del']$, $\Phin$ attains its maximum at $x_n \in [x_0 - \del', x_0]$.  Note that
\[
\Phi_n(x_0 - \del') = v(x_0 - \del') + n(-\del') - v(x_0 - \del') = - n \del' < 0 < v(x_0) = \Phi_n(x_0).
\]
Thus, $x_n \ne x_0 - \del'$, and $\Phi_n$ achieves its maximum on $[x_0 - \del', x_0 + \del']$ at an interior point.  Also, because $x_n \le x_0$,
\[
\dfrac{\rho}{\rho + \bet} \ge v(x_n) \ge \Phi_n(x_n) = v(x_n) + n(x_n - x_0) -\dfrac{v(x_0 - \del')}{(\del')^2} \, (x_n - x_0)^2  \ge v(x_0) > 0,
\]
which implies
\begin{equation}\label{eq:xn_x0}
\lim \limits_{n \to \infty} x_n = x_0.
\end{equation}

By using the viscosity subsolution property of $v$ with the test function
\[
\var(x) = - n(x - x_0) + \frac{v(x_0 - \del')}{(\del')^2} \, (x - x_0)^2
\]
at the point $x_n \in (x_0 - \del', x_0 + \del')$, we have
\[
F \big( (x_n)-, v(x_n), -n + 2v(x_0 - \del') (x_n - x_0)/(\del')^2, v(\cdot) \big) \le 0,
\]
or equivalently,
\begin{align*}
&\bet v(x_n) + \rho( v(x_n) - 1) {\bf 1}_{\{x_n \le 0\}} \\
&+ \sup_R \bigg[ \left(n - \dfrac{2v(x_0 - \del')}{(\del')^2} \, (x_n - x_0) \right) \left(\la \left( (1 + \tet) \E R + \eta \E(YR) - \dfrac{\eta}{2} \, \E \big(R^2 \big) \right) - \kap \right) \\
&\qquad \qquad  - \la \big( \E v(x_n - R) - v(x_n) \big) \bigg] \le 0,
\end{align*}
which implies
\begin{align*}
&\bet v(x_n) + \rho( v(x_n) - 1) {\bf 1}_{\{x_n \le 0\}} - \la \, \dfrac{\rho}{\rho + \bet} \\
&+ \sup_R \bigg[ \left(n - \dfrac{2v(x_0 - \del')}{(\del')^2} \, (x_n - x_0) \right)  \left(\la \left( (1 + \tet) \E R + \eta \E(YR) - \dfrac{\eta}{2} \, \E \big(R^2 \big) \right) - \kap \right) \bigg] \le 0,
\end{align*}
which further implies
\begin{align*}
\bet v(x_n) + \rho( v(x_n) - 1) {\bf 1}_{\{x_n \le 0\}} - \la \, \dfrac{\rho}{\rho + \bet} + c \left(n - \dfrac{2v(x_0 - \del')}{(\del')^2} \, (x_n - x_0) \right)  \le 0,
\end{align*}
which leads to a contradiction as $n \to \infty$ because the last term approaches $+ \infty$; recall the limit in \eqref{eq:xn_x0}.  Thus, we have proved \eqref{eq:vdel_lim}.
\end{proof}

Now, we are ready to prove a comparison theorem.

\begin{thm}\label{thm:comp} {\rm (Comparison principle)}
If $v ~(\text{resp.}, u)$ is a viscosity subsolution $($resp., viscosity supersolution$)$ of \eqref{HJB_F} and \eqref{boundary_condition}, then $v \le u$ on $\R$.
\end{thm}

\begin{proof}
If we apply the standard proof of doubling the variables, we run into difficulties because the operator $F$ is discontinuous at $x = 0$.  To remove this difficulty, we borrow an idea from Giga, G\'orka, and Rybka \cite{GGR2011}, who approximate their discontinuous Hamiltonian with a continuous one.  For some small value of $\del > 0$, define the operator $F^\del$ by
\begin{equation}
\label{oper:F_del}
F^\del \big(x, u(x), v_x(x), w(\cdot) \big) =
\begin{cases}
F\big(x, u(x), v_x(x), w(\cdot) \big), &\quad x < -\del, x \ge 0, \vspace{1.5ex} \\
\left( 1 + \dfrac{x}{\del} \right) \rho \big(1 - u(x) \big) +  F\big(x, u(x), v_x(x), w(\cdot) \big), &\quad -\del \le x < 0.
\end{cases}
\end{equation}
Because the viscosity supersolution $u$ takes values in $\left[ 0, \, \frac{\rho}{\rho + \bet} \right]$, it follows (in the viscosity sense) that
\begin{equation}\label{eq:Fdel_le_Finf}
F^\del \big(x, u(x), u_x(x), u(\cdot) \big) \ge F\big(x, u(x), u_x(x), u(\cdot) \big).
\end{equation}
Thus, $u$ is also a viscosity supersolution of $F^\del = 0$.  Next, from the viscosity subsolution $v$, define the function $v^\del$ by
\begin{equation}\label{eq:vdel}
v^\del(x) = v(x + \del),
\end{equation}
for $x \in \R$.  In Appendix B, we prove that $v^\del$ is a viscosity subsolution of $F^\del = 0$.  In what follows, we prove that $v^\del \le u$ on $\R$.

Define $S$ by
\begin{equation}\label{eq:S}
S = \sup_{x \in \R} \big( v^\del(x) - u(x) \big).
\end{equation}
We wish to show that $S \le 0$; suppose, on the contrary, that $S > 0$.  Note that $S$ is finite because $u$ and $v^\del$ are bounded.  We, next, approximate $S$.  To that end, define $q$ and $\qm$ by \eqref{eq:q} and \eqref{eq:qm}, respectively, with $b$ satisfying
\begin{equation}\label{eq:b}
b > \dfrac{\rho}{\rho + \bet} \, . 
\end{equation}
Define $\Smn$, which we will use to approximate $S$, as follows:
\begin{equation}\label{eq:Smn}
\Smn = \sup_{x, y \in \R} \left(v^\del(x) - u(y) - \qm(x) - \dfrac{n}{2} \, (x - y)^2 \right).
\end{equation}
Because $S > 0$, there exists $x' \in \R$ such that
\[
v^\del(x') - u(x') \ge \dfrac{S}{2} \, .
\]
Let $m > |x'|$, and for the remainder of this proof, assume that $m > |x'|$; then,
\begin{align}\label{eq:approx}
\Smn &= \sup_{x, y \in \R} \left(v^\del(x) - u(y) - \qm(x) - \dfrac{n}{2} \, (x - y)^2 \right) \notag \\
&\ge \sup_{x \in \R} \left(v^\del(x) - u(x) - \qm(x) - \dfrac{n}{2} \, (x - x)^2 \right) \notag \\
&= \sup_{x \in \R} \big(v^\del(x) - u(x) - \qm(x) \big) \notag \\
&\ge \sup_{|x| \le m} \big(v^\del(x) - u(x) - \qm(x) \big) \notag \\
&= \sup_{|x| \le m} \big(v^\del(x) - u(x) \big) \notag \\
&\ge v^\del(x') - u(x') \ge \dfrac{S}{2} \, .
\end{align}
Among other things, we have $\Smn > 0$, which implies that the supremum defining $\Smn$ is achieved on $[-2m, 2m]^2$ because $\qm(x) > ||u||_\infty + ||v^\del||_\infty$ for $|x| \ge 2m$.  Let $(\xmn, \ymn)$ be a point at which the supremum $\Smn$ is realized.  For a fixed value of $m > |x'|$, the sequence $\{(\xmn, \ymn)\}_{n \ge N}$ lies in a bounded region, namely $[-2m, 2m]^2$, which implies that this sequence converges to some point $(x_{m, \infty}, y_{m, \infty})$ as $n$ goes to $\infty$.  Furthermore, the inequality
\begin{equation}\label{eq:approx2}
v^\del(\xmn) - u(\ymn) - \qm(\xmn) - \dfrac{n}{2} \, (\xmn - \ymn)^2 \ge \dfrac{S}{2}
\end{equation}
holds for all $n \in \N$; thus, there exists $C > 0$ and $N \in \N$ such that, for all $n \ge N$,
\begin{equation}\label{eq:diff_bnd}
n (\xmn - \ymn)^2 \le C,
\end{equation}
which implies that $x_{m, \infty} = y_{m, \infty}$.

We can obtain even more from the inequalities in \eqref{eq:approx}.  First, note that
\begin{equation}\label{eq:m_infty}
\lim_{m \to \infty} \sup_{|x| \le m} \big( v^\del(x) - u(x) \big) = \sup_{x \in \R} \big( v^\del(x) - u(x) \big),
\end{equation}
in which the right side equals $S$.  Indeed, because $v^\del - u$ is u.s.c., it follows that, on the set $|x| \le m$, $v^\del - u$ achieves its supremum, say, at $\hat x_m$.  Then, because the interval $[-m, m]$ increases with $m$, the sequence $\{ v^\del(\hat x_m) - u(\hat x_m) \}_{m > |x'|}$ is non-decreasing.  Also, this sequence is bounded above by $S$; therefore, it has a limit $S'$.  Clearly, $S' \le S$, and we wish to show that $S' = S$.  Suppose, on the contrary, that $S' < S$, and define $\del = (S - S')/2$.  By the definition of $S$, there exists $\tilde x$ such that
\[
v^\del(\tilde x) - u(\tilde x) > S - \del = \dfrac{S + S'}{2} > S',
\]
which contradicts the definition of $S'$.  Thus, $S' = S$.

Now, because $\qm \ge 0$, from inequality \eqref{eq:approx}, we have
\[
\sup_{x \in \R} \big( v^\del(x) - u(x) \big) \ge \sup_{x \in \R} \big(v^\del(x) - u(x) - \qm(x) \big) \ge \sup_{|x| \le m} \big(v^\del(x) - u(x) \big),
\]
or equivalently,
\begin{equation}\label{eq:sub_approx}
S \ge \Sm \ge \sup_{|x| \le m} \big(v^\del(x) - u(x) \big),
\end{equation}
in which $\Sm$ denotes the supremum of $v^\del  - u - \qm$ on $\R$.  Then, by taking the limit as $m$ goes to $\infty$ in \eqref{eq:sub_approx} and by using \eqref{eq:m_infty}, we obtain
\[
S \ge \limsup_{m \to \infty} \Sm \ge \liminf_{m \to \infty} \Sm \ge S,
\]
which implies that
\begin{equation}\label{eq:approx_lim}
\lim_{m \to \infty} S_m = S.
\end{equation}
Also, inequality \eqref{eq:approx} implies that $v^\del(\xmn) - u(\ymn) - \qm(\xmn) \ge \Smn \ge \Sm$ for all $n \in \N$; now, let $n$ go to $\infty$ to obtain
\[
S \ge v^\del(x_{m, \infty}) - u(x_{m, \infty}) \ge \limsup_{n \to \infty} \big( v^\del(\xmn) - u(\ymn) - \qm(\xmn) \big) \ge \limsup_{n \to \infty} \Smn \ge \Sm,
\]
in which the second inequality follows because $v^\del - u$ is u.s.c.  Thus, we have
\begin{equation}\label{eq:approx_lim2}
\lim_{m \to \infty} \limsup_{n \to \infty} \Smn = \lim_{m \to \infty} \Sm = S = \lim \limits_{m \to \infty} \big( v^\del(x_{m, \infty}) - u(x_{m, \infty}) \big),
\end{equation}
and $\lim \limits_{m \to \infty} \qm(x_{m, \infty}) = 0$.
Similarly,
\begin{align*}
\Sm &\ge v^\del(x_{m, \infty}) - u(x_{m, \infty}) - \qm(x_{m, \infty}) \\
&\ge \limsup_{n \to \infty} \Big( v^\del(\xmn) - u(\ymn) - \qm(\xmn) - \dfrac{n}{2} \, (\xmn - \ymn)^2 \Big) \\
&= \limsup_{n \to \infty} \Smn \ge \Sm,
\end{align*}
which implies that
\begin{equation}\label{eq:approx_lim3}
\limsup_{n \to \infty} \Smn = \Sm = v^\del(x_{m, \infty}) - u(x_{m, \infty}) - \qm(x_{m, \infty}),
\end{equation}
and $\lim \limits_{n \to \infty} n (\xmn - \ymn)^2 = 0$.

From the definition of $(\xmn, \ymn)$, we deduce that
\begin{equation}
\begin{cases}
\xmn ~ \hbox{is a maximizer of} ~ x \mapsto v^\del(x) - \qm(x) - \dfrac{n}{2} \, (x - \ymn)^2 ~ \hbox{on} ~ \R, \\
\ymn ~ \hbox{is a minimizer of} ~ y \mapsto u(y) + \dfrac{n}{2} \, (\xmn - y)^2 ~ \hbox{on} ~ \R.
\end{cases}
\end{equation}
By using the viscosity subsolution property of $v^\del$ with the test function $\qm(x) + \frac{n}{2} (x - \ymn)^2$ at the point $\xmn$, we obtain
\begin{align*}
F^\del \big(\xmn, v^\del(\xmn), \qm'(\xmn) + n(\xmn - \ymn), v^\del(\cdot) \big) \le 0,
\end{align*}
or equivalently,
\begin{align}
\label{eq:sub_asy}
&\bet v^\del(\xmn) + \min \left(1, - \, \dfrac{\xmn}{\del} \right) \rho( v^\del(\xmn) - 1) {\bf 1}_{\{\xmn < 0\}} + \kap \big( \qm'(\xmn) + n(\xmn - \ymn) \big) \notag \\
&- \la \inf_R \bigg[ \big( \qm'(\xmn) + n(\xmn - \ymn) \big) \left( (1 + \tet) \E R + \eta \E(YR) - \dfrac{\eta}{2} \, \E \big(R^2 \big) \right) \notag \\
&\;\; \qquad \qquad + \E v^\del(\xmn - R) - v^\del(\xmn) \bigg] \le 0.
\end{align}
Similarly, by using the viscosity supersolution property of $u$ with the test function $- \frac{n}{2} (\xmn - y)^2$ at the point $\ymn$, we obtain
\begin{align*}
F^\del \big(\ymn, u(\ymn), n (\xmn - \ymn), u(\cdot) \big) \ge 0,
\end{align*}
or equivalently,
\begin{align}
\label{eq:sup_asy}
&\bet u(\ymn) + \min \left(1, - \, \dfrac{\ymn}{\del} \right) \rho( u(\ymn) - 1) {\bf 1}_{\{\ymn < 0\}} + \kap n(\xmn - \ymn) \notag \\
&- \la \inf_R \left[ n(\xmn - \ymn)\left( (1 + \tet) \E R + \eta \E(YR) - \dfrac{\eta}{2} \, \E \big(R^2 \big) \right) + \E u(\ymn - R) - u(\ymn) \right] \ge 0.
\end{align}
By subtracting inequality \eqref{eq:sup_asy} from \eqref{eq:sub_asy}, we obtain
\begin{align}\label{eq:difference}
& (\la + \bet) \big(v^\del(\xmn) - u(\ymn) \big) + \min \left(1, - \, \dfrac{\xmn}{\del} \right) \rho \big( v^\del(\xmn) - 1 \big) {\bf 1}_{\{\xmn < 0\}} \notag \\
&- \min \left(1, - \, \dfrac{\ymn}{\del} \right) \rho \big( u(\ymn) - 1 \big) {\bf 1}_{\{\ymn < 0\}} + \kap \qm'(\xmn) \notag \\
&- \la \inf_R \left[ \big( \qm'(\xmn) + n(\xmn - \ymn) \big) \left( (1 + \tet) \E R + \eta \E(YR) - \dfrac{\eta}{2} \, \E \big(R^2 \big) \right)  + \E v^\del(\xmn - R) \right] \notag \\
&+ \la \inf_R \left[ n (\xmn - \ymn) \left( (1 + \tet) \E R + \eta \E(YR) - \dfrac{\eta}{2} \, \E \big(R^2 \big) \right)  + \E u(\ymn - R) \right] \le 0.
\end{align}
Note that
\begin{align*}
&\la \inf_R \left[  \E u(\ymn - R) - \E v^\del(\xmn - R) - \qm'(\xmn)  \left( (1 + \tet) \E R + \eta \E(YR) - \dfrac{\eta}{2} \, \E \big(R^2 \big) \right) \right] \\
&\le \la \inf_R \left[ n (\xmn - \ymn) \left( (1 + \tet) \E R + \eta \E(YR) - \dfrac{\eta}{2} \, \E \big(R^2 \big) \right)  + \E u(\ymn - R) \right] \\
&\quad - \la \inf_R \left[ \big( \qm'(\xmn) + n(\xmn - \ymn) \big) \left( (1 + \tet) \E R + \eta \E(YR) - \dfrac{\eta}{2} \, \E \big(R^2 \big) \right) + \E v^\del(\xmn - R) \right].
\end{align*}
The above inequality and \eqref{eq:difference} imply
\begin{align*}
&(\la + \bet) \big(v^\del(\xmn) - u(\ymn) \big) + \min \left(1, - \, \dfrac{\xmn}{\del} \right) \rho \big( v^\del(\xmn) - 1 \big) {\bf 1}_{\{\xmn < 0\}} \\
& - \min \left(1, - \, \dfrac{\ymn}{\del} \right) \rho \big( u(\ymn) - 1 \big) {\bf 1}_{\{\ymn < 0\}} + \kap \qm'(\xmn) \\
&+ \la \inf_R \left[  \E u(\ymn - R) - \E v^\del(\xmn - R) - \qm'(\xmn)  \left( (1 + \tet) \E R + \eta \E(YR) - \dfrac{\eta}{2} \, \E \big(R^2 \big) \right) \right] \le 0,
\end{align*}
or equivalently,
\begin{align}\label{eq:equiv}
&(\la + \bet) \big(v^\del(\xmn) - u(\ymn) \big) + \min \left(1, - \, \dfrac{\xmn}{\del} \right) \rho \big( v^\del(\xmn) - 1 \big) {\bf 1}_{\{\xmn < 0\}} \notag \\
&\quad - \min \left(1, - \, \dfrac{\ymn}{\del} \right) \rho \big( u(\ymn) - 1 \big) {\bf 1}_{\{\ymn < 0\}} + \kap \qm'(\xmn) \notag \\
& \le \la \sup_R \left[\E v^\del(\xmn - R) - \E u(\ymn - R) + \qm'(\xmn)  \left( (1 + \tet) \E R + \eta \E(YR) - \dfrac{\eta}{2} \, \E \big(R^2 \big) \right) \right].
\end{align}
Also, note that
\begin{align*}
&\la \sup_R \left[\E v^\del(\xmn - R) - \E u(\ymn - R) + \qm'(\xmn)  \left( (1 + \tet) \E R + \eta \E(YR) - \dfrac{\eta}{2} \, \E \big(R^2 \big) \right) \right] \\
&\le \la \sup_R \E \big[ v^\del(\xmn - R) - u(\ymn - R) - \qm(\xmn - R) \big]  \\
&\quad + \la \sup_R \left[ \E \qm(\xmn - R) + \qm'(\xmn)  \left( (1 + \tet) \E R + \eta \E(YR) - \dfrac{\eta}{2} \, \E \big(R^2 \big) \right) \right].
\end{align*}
The above inequality and \eqref{eq:equiv} imply
\begin{align*}
&(\la + \bet) \big(v^\del(\xmn) - u(\ymn) \big) + \min \left(1, - \, \dfrac{\xmn}{\del} \right) \rho \big( v^\del(\xmn) - 1 \big) {\bf 1}_{\{\xmn < 0\}} \\
&\quad - \min \left(1, - \, \dfrac{\ymn}{\del} \right) \rho \big( u(\ymn) - 1 \big) {\bf 1}_{\{\ymn < 0\}} \\
& \le \la \sup_R \E \bigg[ v^\del(\xmn - R) - u(\ymn - R) - \qm(\xmn - R) - \dfrac{n}{2} \, (\xmn - \ymn)^2 \bigg] + \dfrac{n}{2} \, (\xmn - \ymn)^2 \\
&\quad + \la \sup_R \left[ \E \qm(\xmn - R) + \qm'(\xmn)  \left( (1 + \tet) \E R + \eta \E(YR) - \dfrac{\eta}{2} \, \E \big(R^2 \big) - \dfrac{\kap}{\la} \right) \right].
\end{align*}
From the definition of $\Smn$ in \eqref{eq:Smn}, we obtain
\begin{align}\label{ineq2}
&\la \big(v^\del(\xmn) - u(\ymn) - \qm(\xmn) \big) + \bet \big(v^\del(\xmn) - u(\ymn) \big) - \dfrac{n}{2} \, (\xmn - \ymn)^2  \notag \\
&\quad + \min \left(1, - \, \dfrac{\xmn}{\del} \right) \rho \big( v^\del(\xmn) - 1 \big) {\bf 1}_{\{\xmn < 0\}} - \min \left(1, - \, \dfrac{\ymn}{\del} \right) \rho \big( u(\ymn) - 1 \big) {\bf 1}_{\{\ymn < 0\}} \notag \\
&\le \la \Smn + \la \left( \sup_R \E \qm(\xmn - R) - \qm(\xmn)  + || \qm' ||_\infty \left\{ (1 + \tet) \E Y + \dfrac{\eta}{2} \, \E \big( Y^2 \big) - \dfrac{\kap}{\la} \right\} \right) .
\end{align}
If we let $n$ go to $\infty$ in \eqref{ineq2}, use the inequalities
\begin{align*}
&\limsup_{n\to \infty} \big(v^\del(\xmn) - u(\ymn) \big) \ge \limsup_{n\to \infty} \big(v^\del(\xmn) - u(\ymn) - \qm(\xmn) \big) \ge \Sm, 
\end{align*}
and cancel the term $\la \Sm$ from each side, we get
\begin{align}\label{ineq3}
&\bet \Sm + \min \left(1, - \, \dfrac{x_{m, \infty}}{\del} \right) \rho \Sm {\bf 1}_{\{x_{m, \infty} < 0\}} \le \la \Big| \Big| \sup_R \, \E \qm(x - R) - \qm(x) \Big| \Big|_\infty + c || \qm' ||_\infty.
\end{align}
By taking a limit as $m$ goes to $\infty$ in \eqref{ineq3} and by using the results of Lemma \ref{lem:qm}, we obtain
\begin{equation}\label{eq:lim_contra}
\bet S + \rho S \liminf_{m \to \infty} \left[ \min \left(1, - \, \dfrac{x_{m, \infty}}{\del} \right) {\bf 1}_{\{x_{m, \infty} < 0 \}} \right] \le 0,
\end{equation}
which contradicts $S > 0$ and $\bet > 0$.

Thus, we have shown that $v^\del \le u$ on $\R$.  By taking the limit superior as $\del \to 0+$, as in \eqref{eq:vdel_lim}, we obtain $v \le u$ on $\R$.
\end{proof}

We now present our main result, an application of the comparison principle in Theorem \ref{thm:comp}.

\begin{thm}\label{thm:value}
The minimum discounted probability of exponential Parisian ruin $\psi$ is the unique viscosity solution of the HJB equation \eqref{HJB_F} with boundary conditions \eqref{boundary_condition}.
\end{thm}

\begin{proof}
From Corollaries \ref{cor:psi_vplus} and \ref{cor:uminus_psi}, we know
\begin{equation}\label{eq:uminus_psi_vplus}
u_{-} \le \psi \le v_{+}
\end{equation}
on $\R$.  Furthermore, Theorems \ref{thm:v_plus} and \ref{thm:u_minus} prove that $v_{+}$ and $u_{-}$ are viscosity sub- and supersolutions, respectively.  Thus, Theorem \ref{thm:comp} implies that $v_{+} \le u_{-}$, which, when combined with \eqref{eq:uminus_psi_vplus} implies that
\[
u_{-} = \psi = v_{+}
\]
on $\R$.  Thus, we have proved this theorem.
\end{proof}

Next, we present a corollary that shows that $\psi$ is differentiable almost everywhere with strictly negative (and finite) derivative where it exists.  We, thereby, deduce that $\psi$ is strictly decreasing on $\R$.

\begin{cor}
The minimum discounted probability of exponential Parisian ruin $\psi$ is differentiable almost everywhere with respect to Lebesgue measure.  Moreover, the lower left-derivative $D_{-} \psi$ defined by
\[
D_{-} \psi(x) = \liminf_{\del \to 0+} \dfrac{\psi(x) - \psi(x - \del)}{\del}
\]
is finite on $\R$.  Finally, $\psi$ is strictly decreasing on $\R$.
\end{cor}

\begin{proof}
Because $\psi$ is non-increasing, Lebesgue's Differentiation Theorem implies the first statement of this corollary.

To prove the second statement, suppose, on the contrary, $D_{-} \psi(x_0) = - \infty$ for some $x_0 \in \R$; then, let $\var \in \mC^1(\R)$ be a test function such that $\psi - \var$ reaches a maximum at $x_0$.  Then, there exists $\del_0 > 0$ such that $0 < \del < \del_0$ implies
\[
\psi(x_0 - \del) - \var(x_0 - \del) \le \psi(x_0) - \var(x_0),
\]
or equivalently,
\[
\dfrac{\var(x_0) - \var(x_0 - \del)}{\del} \le \dfrac{\psi(x_0) - \psi(x_0 - \del)}{\del} \, .
\]
By taking the limit inferior as $\del \to 0+$, we obtain $\var_x(x_0) = -\infty$.  The viscosity subsolution property of $\psi$ implies
\begin{align*}
&\bet \psi(x_0) + \rho \big( \psi(x_0) - 1 \big) {\bf 1}_{\{x_0 \le 0 \}} \\
&+ \sup_R \left[ \left( \kap - \la \left( (1 + \tet) \E R + \eta \E(YR) - \dfrac{\eta}{2} \, \E \big(R^2 \big) \right) \right) \var_x(x_0) + \la \big( \psi(x_0) - \E \psi(x_0 - R) \big) \right] \le 0.
\end{align*}
If we set $R = Y$, then the coefficient of $\var_x(x_0)$ equals $-c$, and we obtain $+\infty \le 0$, a contradiction.  Thus, $D_{-} \psi > - \infty$ on $\R$.

To prove the third statement, suppose, on the contrary, that $\psi$ is not strictly decreasing on $\R$.  Then, because $\psi$ is non-increasing, there exist $x_1 < x_2$ such that, for all $x \in [x_1, x_2]$, $\psi(x_1) = \psi(x) = \psi(x_2)$.  It follows that $\psi_x(x) = 0$ for all $x_1 < x < x_2$.  Let $x_0 \in (x_1, x_2)$ with $x_0 \ne 0$; then, $F \big( x_0, \psi(x_0), 0, \psi(\cdot) \big) = 0$ implies
\[
\bet \psi(x_0) + \rho \big( \psi(x_0) - 1 \big) {\bf 1}_{\{x_0 < 0 \}} = 0,
\]
which contradicts $0 < \psi(x_0) < \frac{\rho}{\rho + \bet}$.  Thus, $\psi$ is strictly decreasing on $\R$.
\end{proof}

\appendix

\section{Equivalence of definitions of viscosity sub- and supersolutions}\label{sec:A}

In this appendix, we prove that Definition \ref{def:strict_glob} is equivalent to Definition \ref{def:local_vis_val}.  In the process, we  introduce another two definitions of viscosity sub- and supersolutions, and we prove those definitions are all equivalent.

\begin{defin}\label{def:glob_vis_tes}
We say a u.s.c.\ function $\underline{u}: \R \to \left[ 0, \, \frac{\rho}{\rho + \bet} \right]$ is a {\rm viscosity subsolution} of \eqref{HJB_F} and \eqref{boundary_condition} if \eqref{boundary_condition} holds and if, for any $x_0 \in \R$ and for any $\var \in \mC^1(\R)$ such that $\underline{u} - \var$ reaches a global maximum of zero at $x_0$, we have
\begin{equation}\label{eq:subA1}
\begin{cases}
F\big(x_0, \underline{u}(x_0), \var_x(x_0), \var(\cdot) \big) \le 0, &\quad x_0 \ne 0, \\
\lim \limits_{x \to 0-} F\big(x, \underline{u}(x), \var_x(x), \var(\cdot) \big) \le 0, &\quad x_0 = 0.
\end{cases}
\end{equation}
Similarly, we say an l.s.c.\ function $\bar{u}: \R \to \left[ 0, \, \frac{\rho}{\rho + \bet} \right]$ is a {\rm viscosity supersolution} of \eqref{HJB_F} and \eqref{boundary_condition} if \eqref{boundary_condition} holds and if, for any $x_0 \in \R$ and for any $\phi \in \mC^1(\R)$ such that $\bar{u} - \phi$ reaches a global minimum of zero at $x_0$, we have
\begin{equation}
F\big(x_0, \bar{u}(x_0), \phi_x(x_0), \bar{u}(\cdot) \big) \ge 0.
\end{equation}
Finally, a function $u$ is called a $($continuous$)$ {\rm viscosity solution} of \eqref{HJB_F} and \eqref{boundary_condition} if it is both a viscosity subsolution and a viscosity supersolution of \eqref{HJB_F} and \eqref{boundary_condition}.  \qed
\end{defin}

\begin{defin}\label{def:glob_vis_val}
We say a u.s.c.\ function $\underline{u}: \R \to \left[ 0, \, \frac{\rho}{\rho + \bet} \right]$ is a {\rm viscosity subsolution} of \eqref{HJB_F} and \eqref{boundary_condition} if \eqref{boundary_condition} holds and if, for any $x_0 \in \R$ and for any $\var \in \mC^1(\R)$ such that $\underline{u} - \varphi$ reaches a global maximum of zero at $x_0$, we have
\begin{equation}\label{eq:subA2}
\begin{cases}
F\big(x_0, \underline{u}(x_0), \varphi_x(x_0), \underline{u}(\cdot) \big) \le 0, &\quad x_0 \ne 0, \\
\lim \limits_{x \to 0-} F\big(x, \underline{u}(x), \var_x(x), \underline{u}(\cdot) \big) \le 0, &\quad x_0 = 0.
\end{cases}
\end{equation}
Similarly, we say an l.s.c.\ function $\bar{u}: \R \to \left[ 0, \, \frac{\rho}{\rho + \bet} \right]$ is a {\rm viscosity supersolution} of \eqref{HJB_F} and \eqref{boundary_condition} if \eqref{boundary_condition} holds and if, for any $x_0 \in \R$ and for any $\phi \in \mC^1(\R)$ such that $\bar{u} - \phi$ reaches a global minimum of zero at $x_0$, we have
\begin{equation}
F\big(x_0, \bar{u}(x_0), \phi_x(x_0), \bar{u}(\cdot) \big) \ge 0.
\end{equation}
Finally, a function $u$ is called a $($continuous$)$ {\rm viscosity solution} of \eqref{HJB_F} and \eqref{boundary_condition} if it is both a viscosity subsolution and a viscosity supersolution of \eqref{HJB_F} and \eqref{boundary_condition}.  \qed
\end{defin}

\begin{prop}
Definitions {\rm \ref{def:glob_vis_tes}} and {\rm \ref{def:glob_vis_val}} are equivalent.
\end{prop}

\begin{proof}
We prove the statement for viscosity subsolutions only; the proof for viscosity supersolutions follows similarly.

\medskip

\noindent {\bf Definition \ref{def:glob_vis_val} $\Rightarrow$ Definition \ref{def:glob_vis_tes}.}  Let $\var \in \mC^1(\R)$, and let $x_0 \in \R$ be such that
\[
\max \limits_{x \in \R} \big\{ (\underline{u} - \var)(x) \big\} = (\underline{u} - \var)(x_0) = 0.
\]
Then, we have
$$
\underline{u}(x_0) - \underline{u}(x) \ge \var(x_0) - \var(x), \quad \hbox{for all}~x \in \R.
$$
It follows that, for $x_0 \ne 0$,
\begin{align*}
&F \big(x_0, \underline{u}(x_0), \var_x(x_0), \var(\cdot) \big) \\
&= \bet \underline{u}(x_0) + \rho \big( \underline{u}(x_0) - 1 \big) {\bf 1}_{\{x_0 < 0\}} + \kap \var_x(x_0)  \\
&\quad - \la \inf_R \left[ \left( (1 + \tet) \E R + \eta \E(YR) - \dfrac{\eta}{2} \, \E \big(R^2 \big) \right) \var_x(x_0) + \E \var(x_0 - R) - \var (x_0) \right] \\
& \le \bet \underline{u}(x_0) + \rho \big( \underline{u}(x_0) - 1 \big) {\bf 1}_{\{x_0 < 0\}} + \kap \var_x(x_0) \\
& \quad - \la \inf_R \left[\left( (1 + \tet) \E R + \eta \E(YR) - \dfrac{\eta}{2} \, \E \big(R^2 \big) \right) \var_x(x_0) + \E \underline{u}(x_0 - R) - \underline{u} (x_0) \right] \\
& = F \big(x_0, \underline{u}(x_0), \var_x(x_0), \underline{u}(\cdot) \big) \le 0,
\end{align*}
in which the last inequality follows from \eqref{eq:subA2} in Definition \ref{def:glob_vis_val}.   Hence, we have proved inequality \eqref{eq:subA1} in Definition \ref{def:glob_vis_tes}.  The case for which $x_0 = 0$ follows similarly by considering the left-limit of $F$ at zero.

\medskip

\noindent {\bf Definition \ref{def:glob_vis_tes} $\Rightarrow$ Definition \ref{def:glob_vis_val}.}  Again, let $\var \in \mC^1(\R)$, and let $x_0 \in \R$ be such that
\[
\max \limits_{x \in \R} \big\{ (\underline{u} - \var)(x) \big\} = (\underline{u} - \var)(x_0) = 0.
\]
Because $\underline{u}$ is bounded, there exists a sequence $\{ \var_n \}_{n \in \N}$ of functions with compact support in $\mC^\infty(\R)$ such that, for an arbitrary $\epsilon_1 > 0$, there exists an $N \in \N$, such that, for all $x \in \R$, $n \ge N$, and $R \in \mathcal{R}$,
\begin{align}
\label{eq:unu_varn}
\int_{\R^+} \big| \underline{u}(x - r) - \var_n(x - r) \big| d F_R(r) \le \epsilon_1,
\end{align}
in which $F_R(r) := \int^\infty_0 {\bf 1}_{\{R(y) \le r\}} \, d F_Y(y)$  is the distribution function of $R$; see Corollary 9.7 in Wheeden and Zygmund \cite{WZ2015}.  Without loss of generality, we may assume that $\underline{u}(x) \le \var_n(x)$ for all $x \in \R$ and $n \in \N$. For each $\veps > 0$ and $n \in \N$, define the function $\Psi^{\veps}_n$ as follows:
\begin{align}
\label{eq:Psi_va}
\Psi^{\veps}_n(x) = \var(x)\,\chi^{\veps}(x) + \var_{n}(x) \big(1-\chi^{\veps}(x) \big), \quad \hbox{for}~x \in \R,
\end{align}
in which $\chi^{\veps} \in \mC^1(\R)$ satisfies
\begin{align*}
\begin{cases}
0 \le \chi^{\veps} (x) \le 1,  &\quad \text{for all} ~ x \in \R, \\
 \chi^{\veps}(x) = 1,   &\quad \text{if} ~ x \in {(x_0 - \veps, x_0 + \veps)},  \\
 \chi^{\veps}(x) = 0,   &\quad \text{if} ~ x \notin {(x_0 - 2\veps, x_0 + 2\veps)}.
 \end{cases}
\end{align*}
Note that $\Psi^{\veps}_n(x_0) = \underline{u}(x_0)$, and
\begin{align*}
\underline{u}(x) - \Psi^{\veps}_n(x) = \big(\underline{u}(x) - \var_n(x)\big) \big(1-\chi^{\veps}(x)\big) + \big(\underline{u}(x) - \var(x) \big)\chi^{\veps}(x) \le 0,
\end{align*}
for all $x \in \R$.  So, $\max \limits_{x \in \R} \big\{ (\underline{u} - \Psi^{\veps}_n)(x) \big\} = \big(\underline{u} - \Psi^{\veps}_n \big)(x_0)$, and $(\Psi^{\veps}_n)_x(x_0) = \var_x(x_0)$. Thus, by Definition \ref{def:glob_vis_tes}, we have
\begin{align*}
 F \big(x_0, \underline{u}(x_0), \var_x(x_0), \Psi^{\veps}_n(\cdot) \big) \le 0.
\end{align*}
Moreover, we have
\begin{align}
& \Big| F \big(x_0, \underline{u}(x_0), \var_x(x_0), \Psi^{\veps}_n(\cdot) \big) -  F \big(x_0, \underline{u}(x_0), \var_x(x_0), \underline{u}(\cdot) \big) \Big| \notag\\
& \quad = \la \Bigg| \inf_R \left[\left( (1 + \tet) \E R + \eta \E(YR) - \dfrac{\eta}{2} \, \E \big(R^2 \big) \right) \var_x (x_0)+ \E \Psi^{\veps}_n(x_0 - R) - \Psi^{\veps}_n (x_0) \right] \notag\\
&\quad \qquad - \inf_R \left[\left( (1 + \tet) \E R + \eta \E(YR) - \dfrac{\eta}{2} \, \E \big(R^2 \big) \right) \var_x(x_0) + \E \underline{u}(x_0 - R) - \underline{u} (x_0) \right]\Bigg| \notag\\
& \quad \le \la \sup_R \Big[ \E \Psi^{\veps}_n(x_0 - R) - \Psi^{\veps}_n (x_0) -\E \underline{u}(x_0 - R) + \underline{u} (x_0) \Big] \notag\\
& \quad = \la \sup_R \bigg[ \int_0^\infty \big( \Psi^{\veps}_n(x_0 - R(y)) - \underline{u} (x_0 - R(y)) \big)d F_Y(y) \bigg]  \notag\\
& \quad \le \la \sup_R \left[ \int_0^\infty \Big\{ \big|\Psi^{\veps}_n(x_0 - R(y)) - \var_n(x_0 - R(y)) \big| + \big| \var_n(x_0 - R(y)) - \underline{u}(x_0 - R(y)) \big| \Big\} dF_Y(y)\right].
\label{eq:com_F}
\end{align}
By combining \eqref{eq:unu_varn}, \eqref{eq:Psi_va} and \eqref{eq:com_F}, and by letting $\epsilon_1 \to 0$ and $\veps \to 0$, we obtain
\begin{align*}
 F \big(x_0, \underline{u}(x_0), \var_x(x_0), \underline{u}(\cdot) \big) \le 0.
\end{align*}
Hence, we have completed our proof.
\end{proof}

In the following proposition, we prove that Definitions \ref{def:strict_glob} and \ref{def:glob_vis_tes} are equivalent, that is, the definition of viscosity sub- and supersolutions does not rely on the maximum or minimum being strict.

\begin{prop}
Definitions {\rm \ref{def:strict_glob}} and {\rm \ref{def:glob_vis_tes}} are equivalent.
\end{prop}

\begin{proof}
We prove the statement for viscosity subsolutions only; the proof for viscosity supersolutions follows similarly.  Because the set of test functions under Definition \ref{def:glob_vis_tes} contains the set of test functions under Definition \ref{def:strict_glob}, it is automatic that, if a function is a viscosity subsolution under Definition \ref{def:glob_vis_tes}, then it is a viscosity subsolution under Definition \ref{def:strict_glob}.

To show the converse, suppose $\underline{u}$ is a viscosity subsolution under Definition \ref{def:strict_glob}.  Let $\var \in \mC^1(\R)$, and let $x_0 \in \R$ be such that $\underline{u} - \var$ reaches a global maximum (not necessarily strict) of zero at $x = x_0$.  In this proof, we only consider the case for which $x_0 \ne 0$.  The case for which $x_0 = 0$ follows similarly by considering the left-limit of $F$ at zero.

For every $\veps > 0$, define the function $\var^\veps$ on $\R$ by
$$
\var^{\veps}(x) = \var(x) + \veps \varpi(x),
$$
in which $\varpi(x) = 1- e^{-(x-x_0)^2}$; then, $\underline{u}(x) \le \var(x) \le \var^{\veps}(x)$ with equality in the second inequality if and only if $x = x_0$.  In other words, $\underline{u} - \var^{\veps}$ has a strict maximum of zero at $x = x_0$.  Hence, by Definition \ref{def:strict_glob},
$$
F \big(x_0, \var^{\veps}(x_0), \var^{\veps}_x(x_0), \var^{\veps}(\cdot) \big) \le 0,
$$
that is,
\begin{align}
\label{eq:var_va}
&\bet \var^\veps(x_0) + \rho \big( \var^\veps(x_0) - 1 \big) {\bf 1}_{\{x_0 < 0\}} + \kap \var^\veps_x (x_0) \notag\\
& - \la \inf_R \left[ \left( (1 + \tet) \E R + \eta \E(YR) - \dfrac{\eta}{2} \, \E \big(R^2 \big) \right) \var^\veps_x (x_0) + \E \var^\veps(x_0 - R) - \var^\veps (x_0) \right] \le 0.
\end{align}
Also, we have
\begin{align}
\label{eq:div_min}
& \inf_R \left[ \left( (1 + \tet) \E R + \eta \E(YR) - \dfrac{\eta}{2} \, \E \big(R^2 \big) \right) \var^\veps_x(x_0) + \E \var^\veps(x_0 - R) - \var^\veps (x_0) \right] \notag\\
& \le \inf_R \left[ \left( (1 + \tet) \E R + \eta \E(YR) - \dfrac{\eta}{2} \, \E \big(R^2 \big) \right) \var_x(x_0) + \E \var(x_0 - R) - \var (x_0) \right]\notag \\
& \quad + \veps \sup_R \left[ \left( (1 + \tet) \E R + \eta \E(YR) - \dfrac{\eta}{2} \, \E \big(R^2 \big) \right) \varpi_x(x_0) + \E \varpi(x_0 - R) - \varpi (x_0) \right].
\end{align}
By combining \eqref{eq:var_va} and \eqref{eq:div_min}, we get
\begin{align*}
&\bet \underline{u}(x_0) + \rho \big( \underline{u}(x_0) - 1 \big) {\bf 1}_{\{x_0 < 0\}} + \kap \big(\var_x(x_0) + \veps \varpi_x(x_0)\big)  \\
&- \la \inf_R \left[ \left( (1 + \tet) \E R + \eta \E(YR) - \dfrac{\eta}{2} \, \E \big(R^2 \big) \right) \var_x(x_0) + \E \var(x_0 - R) - \var (x_0) \right] \\
&- \veps \la \sup_R \left[ \left( (1 + \tet) \E R + \eta \E(YR) - \dfrac{\eta}{2} \, \E \big(R^2 \big) \right) \varpi_x(x_0) + \E \varpi(x_0 - R) - \varpi (x_0) \right] \le 0.
\end{align*}
By letting $\veps$ go to $0$, we obtain
$$
F \big(x_0, \underline{u}(x_0), \var_x(x_0), \var(\cdot) \big) \le 0,
$$
that is, $\underline{u}$ is a viscosity subsolution under Definition \ref{def:glob_vis_tes}, which is what we wished to prove.
\end{proof}

\begin{prop}
Definitions {\rm \ref{def:local_vis_val}} and {\rm  \ref{def:glob_vis_val}} are equivalent.
\end{prop}

\begin{proof}
We prove the statement for viscosity supersolutions only; the proof for viscosity subsolutions follows similarly.  Because the set of test functions under Definition \ref{def:local_vis_val} contains the set of test functions under Definition \ref{def:glob_vis_val}, it is automatic that, if a function is a viscosity supersolution under Definition \ref{def:local_vis_val}, then it is a viscosity supersolution under Definition \ref{def:glob_vis_val}.

To show the converse, suppose $\bar{u}$ is a viscosity supersolution under Definition \ref{def:glob_vis_val}.  Let $\phi \in \mC^1(\R)$, and let $x_0 \in \R$ be such that $\bar{u} - \phi$ reaches a local minimum (not necessarily zero) at $x = x_0$.  Then, there exists an $h > 0$ such that $\bar{u}(x) - \phi(x) \ge \bar{u}(x_0) - \phi(x_0)$ for all $x \in (x_0 - h, x_0 + h)$. By Urysohn's lemma, there is a function $s \in \mC^{\infty}(\R)$ satisfying
\begin{align*}
\begin{cases}
0 \le s (x) \le 1,   &\quad \text{for all} ~ x \in \mathbb{R}, \\
s(x) = 1,  &\quad \text{if} ~ x \in (x_0 - h/2, x_0 + h/2),  \\
s(x) = 0,  &\quad \text{if} ~ x \notin (x_0 - h, x_0 + h).
\end{cases}
\end{align*}
Define the function $\tilde \phi$ on $\R$ by
$$
\tilde{\phi}(x) = s(x)\big(\phi(x) - \phi(x_0) + \bar{u}(x_0) \big) + \big(1 - s(x) \big) H,
$$
in which $H$ is a lower bound of $\bar{u}$ on $\R$.  Because $\bar{u}(x) \ge \phi(x) - \phi(x_0) + \bar{u}(x_0)$ for $x \in (x_0 - h, x_0 + h)$, and $s(x) = 0$ for $x \notin (x_0 - h, x_0 + h)$, we have
\begin{align*}
\bar{u}(x) = s(x) \bar{u}(x) + \big(1 - s(x) \big) \bar{u}(x) \ge s(x) \big(\phi(x) - \phi(x_0) + \bar{u}(x_0) \big) + \big(1 - s(x) \big) H = \tilde{\phi}(x).
\end{align*}
Thus, $\bar{u}(x) \ge \tilde{\phi}(x)$ for all $x\in {\R}$ and $\bar{u}(x_0) = \tilde{\phi}(x_0)$, that is, $\bar{u} - \tilde{\phi}$
reaches a global minimum of zero at $x_0$.  By Definition \ref{def:glob_vis_val}, we have
$$
F\big(x_0, \bar{u}(x_0), \tilde{\phi}_x(x_0), \bar{u}(\cdot) \big) \ge 0.
$$
Because $\tilde{\phi}(x) = \phi(x) - \phi(x_0) + \bar{u}(x_0)$ for $x \in (x_0 - h/2, x_0 + h/2)$,  we have $ \tilde{\phi}_x(x_0)={\phi}_x(x_0)$. Therefore, we obtain
$$
F(x_0, \bar{u}(x_0), {\phi}_x(x_0), \bar{u}(\cdot)) \ge 0,
$$
that is, $\bar{u}$ is a viscosity supersolution under Definition \ref{def:local_vis_val}, which is what we wished to prove.
\end{proof}

\section{Proof that $v^\del$ is a viscosity subsolution of $F^\del = 0$}\label{sec:B}

Let $\var \in \mC^1(\R)$ be a test function such that $v^\del - \var$ reaches a strict, global maximum of $0$ at $x = x_0$.  We want to prove that
\[
F^\del \big( x_0, v^\del(x_0), \var_x(x_0), \var(\cdot) \big) \le 0.
\]
To that end, define $y = x + \del$ and $y_0 = x_0 + \del$.  Also, define a new test function $\var^\del$ by
\[
\var^\del(y) = \var(y - \del),
\]
for $y \in \R$. Then, $\var(x_0) = \var^\del(y_0)$, $\var_x(x_0) = \var^\del_y(y_0)$, and $\E \var(x_0 - R) - \var(x_0) = \E \var^\del(y_0 - R) - \var^\del(y_0)$.

If $x_0 < -\del$ or $x_0 \ge 0$, then
\begin{align*}
&F^\del \big( x_0, v^\del(x_0), \var_x(x_0), \var(\cdot) \big) = F \big( x_0, v(x_0 + \del), \var_x(x_0), \var(\cdot) \big)  \\
&= F \big( y_0 - \del, v(y_0), \var^\del_y(y_0), \var^\del(\cdot) \big) = F \big( (y_0)- , v(y_0), \var^\del_y(y_0), \var^\del(\cdot) \big) \le 0.
\end{align*}
If $-\del \le x_0 < 0$, then
\begin{align*}
&F^\del \big( x_0, v^\del(x_0), \var_x(x_0), \var(\cdot) \big) \\
&= \left( 1 + \dfrac{x_0}{\del} \right) \rho \big(1 - v(x_0 + \del) \big) + F \big( x_0, v(x_0 + \del), \var_x(x_0), \var(\cdot) \big)  \\
&= \left( 1 + \dfrac{x_0}{\del} \right) \rho \big(1 - v(y_0) \big) +  F \big( x_0, v(y_0), \var^\del_y(y_0), \var^\del(\cdot) \big) \\
&= \dfrac{x_0}{\del} \, \rho \big(1 - v(y_0) \big) +  F \big(y_0, v(y_0), \var^\del_y(y_0), \var^\del(\cdot) \big) \le 0.
\end{align*}
Therefore, $v^\del$ is a viscosity subsolution of $F^\del = 0$, which is what we wished to prove.  \qed



\end{document}